\documentclass{amsart}
\usepackage[utf8]{inputenc}
\usepackage{amsfonts}
\usepackage{amssymb}
\usepackage{amsthm}
\usepackage{amsmath}
\usepackage{enumerate}

\theoremstyle{plain}
\newtheorem{theorem}{Theorem}[section]
\newtheorem{lemma}[theorem]{Lemma}
\newtheorem{claim}[theorem]{Claim}
\newtheorem{fact}[theorem]{Fact}
\newtheorem{corollary}[theorem]{Corollary}

\newtheorem{question}[theorem]{Question}
\newtheorem{problem}[theorem]{Problem}
\newtheorem*{ch}{Theorem~\ref{t:characterization}}

\newtheorem*{co}{Theorem~\ref{t:compact family}}
\newtheorem*{cbox}{Theorem~\ref{t:compbox}}
\newtheorem*{pack}{Theorem~\ref{t:comppack}}
\newtheorem*{prob}{Problem~\ref{p:K}}

\theoremstyle{definition}
\newtheorem{definition}[theorem]{Definition}

\newtheorem{notation}[theorem]{Notation}
\newtheorem{remark}[theorem]{Remark}

\numberwithin{equation}{section}

\newcommand{\R}{\mathbb{R}}
\newcommand{\Q}{\mathbb{Q}}
\newcommand{\N}{\mathbb{N}}

\newcommand{\defeq}{\stackrel{\text{def}}{=}}
\newcommand{\eps}{\varepsilon}
\renewcommand{\P}{\mathbb{P}}
\providecommand{\restriction}{\restriction}

\newcommand{\iC}{\mathcal{C}}
\newcommand{\iD}{\mathcal{D}}
\newcommand{\iF}{\mathcal{F}}
\newcommand{\iK}{\mathcal{K}}

\newcommand{\iM}{\mathcal{M}}

\DeclareMathOperator{\length}{length}

\DeclareMathOperator{\diam}{diam}
\DeclareMathOperator{\dist}{dist}

\begin{document}
\title{The range of dimensions of microsets}
\author{Rich\'ard Balka}
\address{Alfr\'ed R\'enyi Institute of Mathematics, Re\'altanoda u.~13--15, H-1053 Budapest, Hungary}
\email{balka.richard@renyi.hu}

\author{M\'arton Elekes}
\address{Alfr\'ed R\'enyi Institute of Mathematics, Re\'altanoda u.~13--15, H-1053 Budapest, Hungary AND E\"otv\"os Lor\'and University, Institute of Mathematics, P\'azm\'any P\'eter s.~1/c, 1117 Budapest, Hungary}
\email{elekes.marton@renyi.hu}
\urladdr{http://www.renyi.hu/$\sim$emarci}

\author{Viktor Kiss}
\address{Alfr\'ed R\'enyi Institute of Mathematics, Re\'altanoda u.~13--15, H-1053 Budapest, Hungary}
\email{kiss.viktor@renyi.hu}

\thanks{The first author was supported by the MTA Premium Postdoctoral Research Program and the National Research, Development and Innovation Office -- NKFIH, grant no.~124749. The second author was supported by the National Research, Development and Innovation Office -- NKFIH, grants no.~124749 and 129211.  The third author was supported by the National Research, Development and Innovation Office -- NKFIH, grants no.~124749, 129211, and~128273.}

\subjclass[2010]{Primary 28A78, 28A80; Secondary 28A05, 82B43}

\keywords{microset, weak tangent, Hausdorff dimension, packing dimension, box dimension, fractal percolation, Hawkes' theorem}

\begin{abstract} 
We say that $E$ is a \emph{microset} of the compact set $K\subset \R^d$ if there exist sequences $\lambda_n\geq 1$ and $u_n\in \R^d$ such that $(\lambda_n K + u_n ) \cap [0,1]^d$ converges to $E$ in the Hausdorff metric, and moreover, $E \cap (0, 1)^d \neq \emptyset$.
	
The main result of the paper is that for a non-empty set $A\subset [0,d]$ there is a compact set $K\subset \R^d$ such that the set of Hausdorff dimensions attained by the microsets of $K$ equals $A$ if and only if $A$ is analytic and contains its infimum and supremum. This answers a question of Fraser, Howroyd, Käenmäki, and Yu. 
	
We show that for every compact set $K\subset \R^d$ and non-empty analytic set $A\subset [0,\dim_H K]$ there is a set $\iC$ of compact subsets of $K$ which is compact in the Hausdorff metric and $\{\dim_H C: C\in \iC \}=A$. The proof relies on the technique of stochastic co-dimension applied for a suitable coupling of fractal percolations with generation dependent retention probabilities.

We also examine the analogous problems for packing and box dimensions.
\end{abstract}
\maketitle

\section{Introduction}

The notion of a microset is due to Furstenberg \cite{Furst}, it was introduced to understand the infinitesimal structure of a compact set $K\subset \R^d$. Microsets are obtained by `zooming' in on $K$ and taking the Hausdorff limit of what we see. For the definition of Hausdorff metric and other notions see the Preliminaries section. 

\begin{definition} Let $d\geq 1$ be a positive integer and let $K\subset \R^d$ be compact. We say that $E$ is a \emph{microset} of $K$ if $E \cap (0, 1)^d \neq \emptyset$ and there exist homotheties $S_n\colon \R^d\to \R^d$ defined as $S_n(z)=\lambda_n z+u_n$ with $\lambda_n\geq 1$ and $u_n\in \R^d$ such that $S_n(K)\cap [0,1]^d$ converges to $E$ in the Hausdorff metric. Then $\{\lambda_n\}_{n\geq 1}$ is called a \emph{scaling sequence} of $E$. 
\end{definition} 

In order not to lose  information about the `thinnest' part of our set $K$, we added the non-standard assumption $E\cap (0,1)^d \neq \emptyset$ to the definition. Without this property $\{0\}$ would be a microset of every Cantor set $K\subset \R$.  Besides this, the definition also varies across different sources. Furstenberg \cite{Furst} allows to converge to $E$ by subsets of $S_n(K)\cap [0,1]^d$, which approach only provides information about the `thickest' part of $K$. Bishop and Peres \cite{BP} assumes $\lambda_n\to \infty$. Fraser, Howroyd, K\"aenm\"aki, and Yu \cite{microsets} define microsets without the restriction $E\cap (0,1)^d\neq \emptyset$, they only add this extra condition to the notion of gallery. Some authors also consider weak tangents, which means that $S_n$ can be arbitrary expanding similarities in the definition instead of homotheties. As the orthogonal group of $\R^d$ is compact, this concept would not significantly differ from ours.

\begin{notation} For a compact set $K\subset \R^d$ define 
\begin{equation*} \iM_K=\{E: E \text{ is a microset of } K \}.
\end{equation*}
\end{notation}	

Let $\dim_H$, $\underline{\dim}_B$, $\overline{\dim}_B$, and $\dim_P$ denote the Hausdorff, lower box, upper box, and packing dimensions, respectively. For the following inequalities see \cite[Page~82]{Ma}.
\begin{fact} \label{f:ineq} For any set $E\subset \R^d$ we have 
\begin{equation*} 
\dim_H E\leq \underline{\dim}_B \, E\leq  \overline{\dim}_B \, E \quad \text{and} \quad 
\dim_H E\leq \dim_P E\leq  \overline{\dim}_B \, E.
\end{equation*}
\end{fact}

The infimum part of the following theorem was proved by Fraser, Howroyd, Käenmäki, and Yu \cite[Theorem~1.1]{microsets}, while the supremum part is basically due to Furstenberg \cite[Theorem~5.1]{Furst}, see also \cite[Theorem~2.4]{microsets}. 

\begin{theorem}[Fraser--Howroyd--K\"aenm\"aki--Yu and Furstenberg] \label{t:Fu}
Let $\dim$ be one of $\dim_H$, $\underline{\dim}_B$, $\overline{\dim}_B$, or $\dim_P$. Assume that $d\geq 1$ and $K\subset \R^d$ is a non-empty compact set. Then $\{\dim E: E\in \iM_K\}$ contains its infimum and supremum. 
\end{theorem}

Fraser, Howroyd, Käenmäki, and Yu \cite[Theorem~1.3]{microsets} proved the following. Recall that a set is $F_\sigma$ if it is a countable union of closed sets.

\begin{theorem}[Fraser--Howroyd--K\"aenm\"aki--Yu] 
Let $d\geq 1$ and let $A\subset [0,d]$ be an $F_{\sigma}$ set which contains its infimum and supremum. Then there exists a compact set $K\subset \R^d$ such that 
\begin{equation*} \{\dim_H E: E\in \iM_K\}=A.
\end{equation*}  
\end{theorem}

They asked the following question \cite[Question~7.4]{microsets}, see also \cite[Question~17.3.2]{F}.

\begin{question}[Fraser--Howroyd--K\"aenm\"aki--Yu] If $K\subset \R^d$ is compact, then is $\{\dim_H E: E\in \iM_K\}$ an $F_{\sigma}$ set? If not, does it belong to a finite Borel class? 
\end{question}

As the main result of this paper, in Section~\ref{s:character} we answer the above questions in the negative. In fact, we prove a complete characterization as follows. 

\begin{ch}[Main Theorem] Let $\dim$ be one of $\dim_H$, $\underline{\dim}_B$, or $\overline{\dim}_B$.
Let $d\geq 1$ and let $A \subset [0, d]$ be a non-empty set. Then the following are equivalent:
\begin{enumerate}
\item There exists a compact set $K \subset \R^d$ such that $\{\dim E : E \in \iM_K\} = A$;
\item $A$ is an analytic set which contains its infimum and supremum;
\item There exists a compact set $K \subset \R^d$ such that $\dim_H E=\overline{\dim}_B \, E$ for all $E\in \iM_K$ and $\{\dim_H E : E \in \iM_K\} = A$.
\end{enumerate} 
\end{ch}
For packing dimension we can only prove one direction due to measurability problems, see Problem~\ref{p:meas}. The following corollary is immediate. 
\begin{corollary}
Let $d\geq 1$ and let $A \subset [0, d]$ be a non-empty analytic 
set which contains its infimum and supremum. Then there exists a compact set $K \subset \R^d$ with 
\begin{equation*}
\{\dim_P E : E \in \iM_K\} = A.
\end{equation*}
\end{corollary}

It is natural to consider microsets of subsets of a given compact set $K\subset \R^d$.

\begin{prob} 
Let $K\subset \R^d$ be compact. Characterize the sets $A\subset [0,d]$ for which there exists a compact set $C\subset K$ such that $\{\dim_H E: E\in \iM_C\}=A$.
\end{prob}

However, we cannot solve this more general problem, since we are unable to control subsets and microsets at the same time. Actually, we do not even have a conjecture for the characterization. As the set $\iM_C$ is always $\sigma$-compact, considering compact families of compact subsets of $K$ seems to be the right first step in this direction. Let $\iK(K)$ denote the set of non-empty compact subsets of $K$ endowed with the Hausdorff metric. Subsection~\ref{ss:Haus} is dedicated to the proof of the following theorem, which relies on the stochastic co-dimension method applied for a suitable coupling of fractal percolations with generation dependent retention probabilities.

\begin{co}
Let $K \subset \R^d$ be a non-empty compact set and let $A \subset [0, \dim_H K]$. Then the following are equivalent: 
\begin{enumerate}
\item There is a compact set $\mathcal{C} \subset \mathcal{K}(K)$ with $\{\dim_H C : C \in \mathcal{C}\} = A$;
\item $A$ is an analytic set. 
\end{enumerate} 
\end{co}

In Subsection~\ref{ss:box} we consider the analogous problems for box and packing dimensions. In case of lower box dimension the analogue of Theorem~\ref{t:compact family} does not hold. The following theorem of Feng, Wen, and Wu \cite[Theorem~3]{FWW} demonstrates that $\underline{\dim}_B$ even fails to satisfy the Darboux property. 

\begin{theorem}[Feng--Wen--Wu] There exists a compact set $K\subset [0,1]$ such that $\underline{\dim}_B \, K=1$ and any set $E\subset K$ satisfies $\underline{\dim}_B \, E\in \{0,1\}$. 
\end{theorem}

For upper box dimension an analogous version of Theorem~\ref{t:compact family} holds. 

\begin{cbox} Let $K$ be a non-empty compact metric space and $A \subset [0, \overline{\dim}_B \, K]$. Then the following are equivalent: 
\begin{enumerate}
\item There is a compact set $\mathcal{C} \subset \mathcal{K}(K)$ with $\{\overline{\dim}_B \, C : C \in \mathcal{C}\} = A$;
\item $A$ is an analytic set. 
\end{enumerate} 
\end{cbox}

For packing dimension we will show the following.

\begin{pack}
Let $K$ be a non-empty compact metric space and let $A \subset [0, \dim_P K]$ be analytic. Then there is a compact set $\mathcal{C} \subset \mathcal{K}(K)$ with $\{\dim_P C : C \in \mathcal{C}\} = A$.
\end{pack}

\begin{remark}
Since every analytic set $B\subset \R^d$ contains a compact set $C$ such that $\dim_H C=\dim_H B$ and $\dim_P C=\dim_P B$, Theorems~\ref{t:compact family} and \ref{t:comppack} remain true for compact families of compact subsets of a given analytic set $B\subset \R^d$ as well.
\end{remark}

Finally, we collect the open problems in Section~\ref{s:open}.

\section{Preliminaries} 
Let $d\geq 1$ be an integer. A set $A\subset \R^d$ is \emph{analytic} if there exist a Polish space $X$ and a continuous onto map $f\colon X\to A$. Let $(X,\rho)$ be a Polish space or compact metric space. Let $(\mathcal{K}(X),d_{H})$ be the set of non-empty compact subsets of
$X$ endowed with the \emph{Hausdorff metric}, that is, for each $K_1,K_2\in \mathcal{K}(X)$ we have
\begin{equation*} 
d_{H}(K_1,K_2)=\min \left\{r: K_1\subset B(K_2,r) \textrm{ and } K_2\subset B(K_1,r)\right\},
\end{equation*}
where $B(A,r)=\{x\in X: \exists y\in A \textrm{ such that } \rho(x,y)|\leq r\}$.
Then $(\mathcal{K}(X),d_{H})$ is a Polish space, and it is compact when $X$ is compact, see \cite[Theorems~4.25,~4.26]{Ke}.

Let $X$ be a metric space. For $x\in X$ and $r>0$ let $B(x,r)$ be the closed ball of radius $r$ centered at $x$. For every $s\geq 0$ the \emph{$s$-Hausdorff content} of $X$ is defined as
\begin{equation*} \mathcal{H}^{s}_{\infty}(X)=\inf \left\{ \sum_{i=1}^\infty (\diam
E_{i})^{s} : X\subset \bigcup_{i=1}^{\infty} E_{i} \right\},
\end{equation*}
where $\diam E_i$ denotes the diameter of $E_i$. The \emph{Hausdorff dimension} of $X$ is
\begin{equation*} \dim_{H} X = \inf\{s \ge 0: \mathcal{H}_{\infty}^{s}(X) =0\}.
\end{equation*}
Let $N_n(X)$ be the minimal number of closed balls of radius at most $2^{-n}$ needed to cover $X$. The \emph{lower box dimension} and \emph{upper box dimension} of $X$ are defined as
\begin{equation*}  \underline{\dim}_{B}\, X=\liminf_{n \to \infty} \frac{\log N_{n}(X)}{n\log 2} \quad \text{and} \quad \overline{\dim}_{B} \, X=\limsup_{n\to \infty} \frac{\log N_{n}(X)}{n\log 2},
\end{equation*} 
respectively. The \emph{packing dimension} of $X$ is defined by
\begin{equation*} \dim_P X=\inf \left\{\sup_{i} \overline{\dim}_{B} \, E_i: X=\bigcup_{i=1}^{\infty} E_i\right\}.
\end{equation*}
For the following lemma see \cite[Lemma~2.8.1]{BP}.
\begin{lemma} \label{l:packing}
Let $X$ be a non-empty separable metric space. 
\begin{enumerate}[(i)]
\item \label{i:i} If $X$ is complete and $\overline{\dim}_B \, U\geq \alpha$ for each non-empty open set $U\subset X$, then $\dim_P X\geq \alpha$.
\item \label{i:ii} If $\dim_P X>\alpha$ then there is a closed subset $F\subset X$ such that $\dim_P (F\cap U)>\alpha$ for each open set $U$ intersecting $F$.
\end{enumerate}
\end{lemma}

The next fact easily follows from the finite stability of the dimensions in question. 

\begin{fact} \label{f:dimK}
Let $\dim$ be one of $\dim_H$, $\overline{\dim}_B$, or $\dim_P$. Let $K$ be a non-empty compact metric space. Then there exists $z\in K$ such that $\dim B(z,r)=\dim K$ for all $r>0$. 
\end{fact}

For the next claim see \cite[Product formulas 7.2 and 7.5]{Fa}.
\begin{claim} \label{c:product}
Let $A\subset \R^d$ and $B\subset \R^m$ be compact sets. Then 
\begin{equation*} \dim_H A+\dim_H B\leq \dim_H(A\times B)\leq  \overline{\dim}_B \, (A\times B) \leq  \overline{\dim}_B \, A+\overline{\dim}_B \, B.
\end{equation*}  
\end{claim}
\noindent Consult~\cite{Fa} or \cite{Ma} for more on these concepts.

\bigskip

For any $x\in 2^{\omega}$ define the compact set $K(x)\subset [0,1]$ as
\begin{equation*} 
K(x)=\left\{\sum_{i=0}^{\infty} a_i 2^{-i-1}: a_i=0 \text{ if } x(i)=0 \text{ and } a_i\in \{0,1\} \text{ if } x(i)=1\right\}.
\end{equation*}
For $n\in \omega$ let $x\restriction n$ be the restriction of $x$ to its first $n$ coordinates. We endow $2^{\omega}$ with a metric $d$ compatible with the product topology defined as 
\begin{equation*} d(x,y)=2^{-\min\{i: x(i)\neq y(i)\}} \quad \text{for all } x,y\in 2^{\omega}.
\end{equation*}
The following fact is straightforward. 
\begin{fact} \label{f:cont} 
The map $K\colon 2^{\omega}\to \iK([0,1])$ mapping $x$ to $K(x)$ is continuous, more precisely, 
\begin{equation*} d_H(K(x),K(y))\leq d(x,y)  \quad \text{for all } x,y\in 2^{\omega}.
\end{equation*}
\end{fact} 
Define the left shift $T\colon 2^{\omega} \to  2^{\omega}$ as
\begin{equation*} 
T(x)(n)=x(n+1) \text{ for all } n\in \omega.
\end{equation*} 
We define the \emph{lower density} and \emph{upper density of $x$} by
\begin{equation*} 
\underline{\varrho}(x)=\liminf_{n \to \infty} \frac{\sum_{i=0}^{n-1} x(i)}{n} \quad \text{and} \quad \overline{\varrho}(x)=\limsup_{n \to \infty} \frac{\sum_{i=0}^{n-1} x(i)}{n},
\end{equation*}    
respectively. If $\underline{\varrho}(x)=\overline{\varrho}(x)$ then the common value $\varrho(x)$ is called the \emph{density of $x$}. For the following claim see e.g.~\cite[Example~3.2.3]{BP}.
\begin{claim} \label{cl:1}
For any $x\in 2^{\omega}$ we have 
\begin{equation*} 
\dim_H K(x)=\underline{\dim}_B \, K(x)=\underline{\varrho}(x) \quad \text{and} \quad \dim_P K(x)=\overline{\dim}_B \, K(x)=\overline{\varrho}(x).
\end{equation*}
\end{claim}
Claims~\ref{c:product} and \ref{cl:1} allow us to calculate the dimensions of the Cartesian product $K(x)^d$ as follows. 
\begin{fact} \label{f:dim}
Let $\dim$ be one of $\dim_H$, $\underline{\dim}_B$, $\overline{\dim}_B$ or $\dim_P$. Assume that $d \ge 1$ and $x\in 2^{\omega}$ are given such that $\varrho(x)$ exists. Then 
\begin{equation*} \dim K(x)^d=d\varrho(x). 
\end{equation*} 
\end{fact}

We denote the set of finite binary sequences by $2^{<\omega}$. For $s,t\in 2^{<\omega}$ let $s^\frown t\in 2^{<\omega}$ denote the concatenation of $s$ and $t$. Let $\lfloor \cdot \rfloor$ denote the floor function.

\section{Range of dimensions of microsets} \label{s:character}

The goal of this section is to prove Theorem~\ref{t:characterization} after some preparation. 

\subsection{A description of the microsets of $K(x)^d$} The goal of this subsection is to prove Theorem~\ref{t:Cx}, which provides a useful tool to work with the microsets of $K(x)^d$. For technical reasons we generalize $\iM_K$ as follows. 

\begin{definition}
	For $d\geq 1$ and $\iF\subset \iK(\R^d)$ define $\iM(\iF)$ as the set of compact sets $E\subset [0,1]^d$ for which $E\cap (0,1)^d\neq \emptyset$ and there exist $K_n\in \iF$, $\lambda_n\geq 1$, and $u_n\in \R^d$ such that $(\lambda_n K_n+u_n)\cap [0,1]^d\to E$. 
\end{definition}	

\begin{fact} \label{f:subset} Let $d\geq 1$ and assume that $C_n\subset K_n\subset \R^d$ are compact sets such that $C_n\to C$ and $K_n\to K$. Then $C\subset K$.
\end{fact}

\begin{definition}
	For $x\in 2^{\omega}$ and $n\in \omega$ define the finite set
	\begin{equation*} 
	F_n(x)=\left\{\sum_{i=0}^{n-1} a_i 2^{-i-1}: a_i=0 \text{ if } x(i)=0 \text{ and } a_i\in \{0,1\} \text{ if } x(i)=1\right\}
	\end{equation*}
	and let
	\begin{equation*} 
	\iD_n(x)=\{2^{-n}K(T^n(x))+u: u\in F_n(x)\}.
	\end{equation*}
	Clearly, for each $n$ we have $\bigcup \iD_n(x)=K(x)$ and elements of $\iD_n(x)$ have pairwise non-overlapping convex hulls.
\end{definition}

The next lemma explains that microsets of the family $\{K(x): x\in 2^{\omega}\}$ are `not very far' from being just the sets of the form $K(x)$.

\begin{lemma} \label{l:main} 
	Let $E\in \iM(\{K(x): x\in 2^{\omega}\})$. Assume $(\lambda_n K(x_n)+u_n)\cap [0,1]\to E$ for some $\lambda_n\geq 1$, $x_n\in 2^{\omega}$, and $u_n\in \R$. Then there exist $x\in 2^{\omega}$, $c\in \R^+$, $m_n\in \omega$ for all $n$, a subsequence of positive integers $k_n\uparrow \infty$, a similar copy $C(x)$ of $K(x)$, $w_0=0$ and $w_1,w_2,w_3\in \R$ such that 
	\begin{enumerate} 
		\item \label{e:T1} $T^{m_n}(x_{k_n})\to x$, 
		\item \label{e:T2} $\lambda_{k_n}2^{-m_n}\to c$, 
		\item \label{e:T3} $C(x)\subset E\subset \bigcup_{i=0}^3 (C(x)+w_i)$.
	\end{enumerate}
\end{lemma}
\begin{proof}
Define $C^i_n=2^{-i} K(T^i(x_n))$ for all $i,n\in \omega$. Let $\eps\in (0,1/2)$ be such that $E\cap (\eps,1-\eps)\neq \emptyset$. For all $n$ let $\varphi_n \colon \R  \to \R$ be defined as $\varphi_n(z)=\lambda_n z+u_n$ and let $p_n\in \omega$ be the minimal integer such that there is a $v^0_n\in F_{p_n}(x_n)$ such that 
\begin{equation} \label{e:vn} \varphi_n (C_n^{p_n}+v^0_n) \subset [0,1].
\end{equation} 
By the minimality of $p_n$, for all $n$ there are at most $4$ translations $v\in F_{p_n}(x_n)$ satisfying 
\[\varphi_n(C_n^{p_n}+v)\cap [0,1] \neq \emptyset,\] 
assume that they are $v^0_n,\dots,v_n^{\ell_n}$ for some $\ell_n\in\{0,1,2,3\}$. Then for all $n\geq 1$ we obtain 
\begin{equation} \label{e:pni}
\varphi_n(K(x_n))\cap [0,1]=\left( \bigcup_{i=0} ^{\ell_n} \varphi_n(C_n^{p_n}+v_n^i) \right)\cap [0,1].
\end{equation}
	
	First we show $\lambda_n 2^{-p_n}\geq \eps/2$ for all $n$ large enough. Indeed, if $n$ is large enough then $\varphi_n (K(x_n))\cap (\eps,1-\eps)\neq \emptyset$. Thus for $k=p_n-1$ there is a $D\in \iD_{k}(x_n)$ with $\varphi_n(D)\cap (\eps,1-\eps)\neq \emptyset$. Assume to the contrary that $\lambda_n 2^{-p_n}<\eps/2$. Then $\diam \varphi_n(D)\leq \lambda_n 2^{-k}<\eps$,  so $\varphi_n(D)\subset [0,1]$, contradicting the minimality of $p_n$.
	
Now let $q_n\geq p_n$ be the minimal integer for which $\lambda_n 2^{-q_n}\leq 2$. We prove $x_n(i)=0$ for all $i\in \{p_n,\dots,q_n-1\}$. Assume to the contrary that this is not the case, and take the minimal $i$ such that $p_n\leq i<q_n$ and $x_n(i)=1$. As $x_n(j)=0$ for all $p_n\leq j<i$, we easily obtain $C_n^i=C_n^{p_n}$. Hence $\varphi_n (C_n^{i} + v^0_n)\subset [0,1]$. Since $x_n(i)=1$, we have $\diam K(T^i(x_n))\geq 1/2$, so 
	\begin{equation*} 
	\lambda_n 2^{-i-1}\leq \lambda_n 2^{-i} \diam K(T^i(x_n))=\diam \varphi_n (C_n^{i}) = \diam \varphi_n (C_n^{i} + v^0_n) \leq 1.
	\end{equation*} 
	Hence $\lambda_n 2^{-i}\leq 2$, which contradicts the minimality of $q_n$. As $x_n(i)=0$ for all $i\in \{p_n,\dots,q_n-1\}$, we obtain 
	\begin{equation} \label{e:pnqn}
	    C_n^{q_n}=C_n^{p_n} \text{ for all } n\in \omega.
	\end{equation}
Since $\lambda_n 2^{-q_n}\in [\eps/2,2]$ for all large enough $n$, we can choose $c\in [\eps/2,2]$ and a subsequence $k_n\uparrow \infty$ such that $\ell_{k_n}=\ell\in \{0,1,2,3\}$ for all $n$ and $m_n=q_{k_n}$ satisfies $\lambda_{k_n} 2^{-m_n} \to c$, so \eqref{e:T2} holds.
	
	As $2^{\omega}$ is compact, by choosing a subsequence we may assume that there exists $x\in 2^{\omega}$ such that 
	$T^{m_n}(x_{k_n})\to x$,
	so \eqref{e:T1} holds as well. 
	
Recall that $m_n=q_{k_n}$. An application of \eqref{e:pnqn} for $k_n$, \eqref{e:vn}, and Fact \ref{f:cont} imply that we may assume by choosing a subsequence that there exists $v\in \R$ such that
	\begin{equation} \label{e:kn1} 
	\varphi_{k_n}(C^{m_n}_{k_n}+v^0_{k_n})\to C(x) \defeq cK(x)+v\subset [0,1].
	\end{equation} 
As $\varphi_{k_n}(C^{m_n}_{k_n}+v^i_{k_n})\subset [-1,2]$ is isometric to $\varphi_{k_n}(C^{m_n}_{k_n}+v^0_{k_n})$ for every $1\leq i\leq \ell$, we may suppose by choosing a subsequence that 
	\begin{equation} \label{e:kn2}  
	\varphi_{k_n}(C^{m_n}_{k_n}+v^i_{k_n})\to C(x)+w_i 
	\end{equation} 
	for all $1\leq i\leq \ell$ with some $w_1,\dots,w_{\ell}\in \R$. An application of the formulas \eqref{e:pnqn}, \eqref{e:pni}, and $\varphi_n(K(x_n))\cap [0,1]\to E$ for the subsequence $k_n$ implies 
	\begin{equation} \label{e:kn3} 
	\left( \bigcup_{i=0}^{\ell} \varphi_{k_n}(C^{m_n}_{k_n}+v^i_{k_n}) \right)\cap [0,1]\to E.
	\end{equation}
	Let $w_0=0$. Fact~\ref{f:subset} and \eqref{e:kn1}, \eqref{e:kn2}, and \eqref{e:kn3} imply that 
	\begin{equation*}
	C(x)\subset E\subset \bigcup_{i=0}^{\ell} (C(x)+w_i),
	\end{equation*}
	so \eqref{e:T3} holds no matter how we define $w_i$ for $i>\ell$. The proof is complete.
\end{proof}

\begin{theorem} \label{t:Cx}
	Let $d\geq 1$ and let $E\in \iM(\{K(x)^d: x\in 2^{\omega}\})$. Assume that $(\lambda_n K(x_n)^d+u_n)\cap [0,1]^d \to E$ for some $\lambda_n\geq 1$, $x_n\in 2^{\omega}$, and $u_n\in \R^d$. Then there exist $x\in 2^{\omega}$, $m_n\in \omega$ for all $n$, a subsequence of positive integers $k_n\uparrow \infty$, a similar copy $C(x)$ of $K(x)$, and $v_1,\dots,v_\ell\in \R^d$ such that 
	\begin{enumerate}[(i)]
		\item \label{e:d1} $T^{m_n}(x_{k_n})\to x$,
		\item \label{e:d2} $C(x)^d+v_1\subset E\subset  \bigcup_{i=1}^\ell (C(x)^d+v_i)$. 
	\end{enumerate}
\end{theorem}
\begin{proof} Let $u_n=(u_n^1,\dots, u_n^d)$ for all $n$. It easily follows that $E=E_1\times\dots \times E_d$, where $E_i\subset [0,1]$ are compact sets such that $E_i\cap (0,1)\neq \emptyset$ and 
	\begin{equation*} 
	(\lambda_n K(x_n)+u^i_n)\cap [0,1] \to E_i \text{ as } n\to \infty
	\end{equation*}
	for all $1\leq i\leq d$. Applying Lemma~\ref{l:main} for all $1\leq i\leq d$ successively implies that there exist $z_i\in 2^{\omega}$, $c_i\in \R^+$, $m_{i,n}\in \omega$ for all $n$, a subsequence of positive integers $k_{n}\uparrow \infty$ (note that this does not depend on $i$), a similar copy $C(z_i)$ of $K(z_i)$, and $w_{i,j}\in \R$ for $0\leq j\leq 3$ such that 
	\begin{enumerate} 
		\item \label{e:dd1} $T^{m_{i,n}}(x_{k_n})\to z_i$, 
		\item \label{e:dd2} $\lambda_{k_n}2^{-m_{i,n}}\to c_i$, 
		\item \label{e:dd3} $C(z_i)\subset E_i\subset  \bigcup_{j=0}^3 (C(z_i)+w_{i,j})$. 
	\end{enumerate}
	By \eqref{e:dd2} for large enough $n$ and for all $1\leq i\leq j\leq d$ we obtain that $m_{i,n}-m_{j,n}$ is independent of $n$. We may assume that $m_{1,n}-m_{i,n}=\ell_i\in \N$ for any $1\leq i\leq d$ and any large enough $n$. Let $m_n=m_{1,n}$ and $x=z_1$, then clearly \eqref{e:d1} holds. By \eqref{e:dd1} we obtain $x=T^{\ell_i}(z_i)$ for all $1\leq i\leq d$. Therefore, $C(z_i)$ is a union of at most $2^{\ell_i}$ many translates of $C(x)$ for all $1\leq i\leq d$, so taking the product of \eqref{e:dd3} for all $i\in \{1,\dots,d\}$ implies \eqref{e:d2} with $\ell= \prod_{i=1}^d 2^{\ell_i+2}$, which finishes the proof.
\end{proof}

\subsection{Useful lemmas for analytic sets and balanced sequences} The goal of this subsection is to prove Lemmas~\ref{l:varphi existence} and \ref{l:balanced}.

\begin{fact} \label{f:Gd}
Let $A \subset [0,\infty)$ be a non-empty analytic set. Then there exist a $G_\delta$ set $G \subset 2^\omega$ and a continuous map $f \colon G \to [0,\infty)$ such that $f(G) = A$.
\end{fact} 
\begin{proof} 
Let $g \colon 2^\omega \to [0,1]$ be a continuous surjection, and let $\psi \colon [0, 1) \to [0, \infty)$ be a homeomorphism. Since $g^{-1}(\psi^{-1}(A)) \subset 2^\omega$ is analytic, by \cite[Exercise~14.3]{Ke} we can find a $G_\delta$ set $H \subset 2^\omega \times 2^\omega$ such that $\pi_1(H) = g^{-1}(\psi^{-1}(A))$, where $\pi_1$ denotes the projection onto the first coordinate. Let $h\colon 2^{\omega } \to 2^\omega \times 2^\omega$ be a homeomorphism. As $g(g^{-1}(\psi^{-1}(A))) = \psi^{-1}(A)$, taking $G=h^{-1}(H)$ and $f=\psi \circ g \circ \pi_1\circ h|_{G}$ finishes the proof. 
\end{proof} 

\begin{definition} 
For $s\in 2^{<\omega}$ we denote by $\length(s)$ the number of coordinates of $s$, where $\length(\emptyset)=0$. Define 
\begin{equation*} 
[s]=\{x\in 2^{\omega}: x\restriction \length(s)=s\}.
\end{equation*} 
\end{definition} 

\begin{lemma}
\label{l:varphi existence}
Let $A \subset [0,\infty)$ be a non-empty analytic set. Then there exists a map $\varphi \colon 2^{<\omega} \to [0, \infty)$ such that
\begin{equation*} 
\overline{\varphi}(x) = \lim_{n \to \infty} \varphi(x \restriction n)
\end{equation*} 
exists for each $x \in 2^\omega$, and the resulting function $\overline{\varphi}$ satisfies $\overline{\varphi}(2^\omega) = A$. 
\end{lemma}
\begin{proof}
According to Fact~\ref{f:Gd} we can choose a $G_\delta$ set $G \subset 2^\omega$ and a continuous map $f \colon G \to [0, \infty)$ such that $f(G) = A$. The set $F = 2^\omega \setminus G$ is $F_{\sigma}$, so it can be written as $F = \bigcup_{n=1}^{\infty} F_n$ where $F_n\subset 2^{\omega}$ are closed and $F_n \subset F_{n+1}$ for all $n\geq 1$. We define $\varphi$ on an element $s \in 2^{<\omega}$ by induction on the length of $s$. 
Fix $a_0\in A$ arbitrarily and define $\varphi(\emptyset) = a_0$. Now suppose that $\varphi$ is already defined on $s \in 2^{<\omega}$, our task is to define it on $s^\frown c$ where $c \in \{0, 1\}$. 
If $[s^\frown c] \cap F = \emptyset$ then let $\varphi(s^\frown c)$ be an arbitrary element of $f([s^\frown c])$. If $[s^\frown c] \cap G = \emptyset$ then let $\varphi(s^\frown c) = \varphi(s)$. 
	
It remains to define $\varphi(s^\frown  c)$ if $[s^\frown  c]$ intersects both $G$ and $F$. Let $m(s^\frown  c)$ and $m(s)$ be the smallest indices with $[s^\frown  c] \cap F_{m(s^\frown  c)} \neq \emptyset$ and $[s] \cap F_{m(s)} \neq \emptyset$, respectively. If $m(s^\frown  c) = m(s)$ then let $\varphi(s^\frown  c) = \varphi(s)$. Otherwise, let $\varphi(s^\frown  c)$ be an arbitrary element of $f([s^\frown  c] \cap G)$, concluding the definition of $\varphi$. 
	
It remains to check that $\varphi$ satisfies the conditions of the lemma. First note that 
\begin{equation} \label{e:varphi s in A}
\varphi(s) \in A \text{ for each } s \in 2^{<\omega},
\end{equation}
a fact that can be quickly checked by induction. Let $x \in 2^\omega$ be fixed. It is enough to show that $\overline{\varphi} (x) = \lim_{n \to \infty} \varphi(x \restriction n)$ exists, $\overline{\varphi}(x) \in A$, and if $x \in G$ then $\overline{\varphi}(x) = f(x)$. 
	
First assume that $[x \restriction m] \cap F = \emptyset$ for some $m$. Then $\varphi(x \restriction n)$ is an element of $f([x \restriction n])$ for each $n \ge m$. Hence, using the continuity of $f$ and the fact that $x \in G$, we obtain $\varphi(x\restriction n) \to f(x) \in A$. 
	
Now suppose that $[x \restriction m] \cap G = \emptyset$ for some $m$. The definition of $\varphi$ and \eqref{e:varphi s in A} imply that there exists $a \in A$ such that $\varphi(x \restriction n) = a$ for all $n\geq m$. It follows that $\varphi(x \restriction n) \to a \in A$. Note that in this case $x \in F$, so we do not have to check that $\overline{\varphi}(x) = f(x)$.
	
Finally, assume that $[x \restriction n]$ intersects both $G$ and $F$ for each $n$. Clearly, $m(x \restriction n)$ increases as $n \to \infty$. Suppose first that $m(x \restriction n) \to \infty$. Then $x \not \in F_n$ for each $n$, hence $x \in G$. Also, for infinitely many $n$, the value of $\varphi(x \restriction n)$ is chosen from $f([x \restriction n])$, and when it is not, then $\varphi(x \restriction n) = \varphi(x \restriction (n - 1))$. It follows that $\varphi(x \restriction n) \to f(x) \in A$. Finally, assume that $m(x \restriction n) $ converges, that is, there exists $m$ such that $m(x \restriction n)= m$ for all large enough $n$. Then $[x \restriction n] \cap F_m \neq \emptyset$ for each $n$, hence $x \in F_m \subset F$, so we do not need to check $\overline{\varphi}(x) = f(x)$. Using \eqref{e:varphi s in A} it also follows that there exists $a \in A$ such that $\varphi(x \restriction n) = a$ if $n$ is large enough. The proof is complete.
\end{proof}

\begin{definition}
 We call $s\in 2^{<\omega}$ a \emph{segment} of $x\in 2^{\omega}$ or $t\in 2^{<\omega}$ if \begin{equation*} 
 s = (x(k), x(k+1), \dots, x(k+n-1)) \text{ or } s = (t(k), t(k+1), \dots, t(k+n-1))
 \end{equation*} 
for some $k, n\in \omega$ with $\length(t) \ge k+n$. Such a relation is denoted by $s \sqsubseteq x$ and $s \sqsubseteq t$, respectively. Let $I\subset \omega$ be a \emph{discrete interval} if $I=\{k,\dots,k+n-1\}$ for some $k,n\in \omega$, then define $x\restriction I=(x(k),\dots,x(k+n-1))$. For two discrete intervals $I,J\subset \omega$ we write $I<J$ if $\max I<\min J$. For $s\in 2^{<\omega}\setminus \{\emptyset\}$ let 
\begin{equation*} \sigma(s)=\sum_{i=0}^{\length(s)-1} s(i) \quad \text{and} \quad  \varrho(s)=\frac{\sigma(s)}{\length(s)}.
\end{equation*}
The sequence $x\in 2^{\omega}$ is called \emph{balanced} if for all $n$ for any two segments $s,t \sqsubseteq x$ of length $n$ we have $|\sigma (s) - \sigma (t)|\le 1$. It is known that for each $a \in [0, 1]$ there exists a balanced sequence $x\in 2^{\omega}$ with $\varrho(x) =a$, see e.g.~\cite[Section~2.2]{balanced} for the details.
\end{definition}

\begin{lemma}
\label{l:balanced}
Let $0\leq a\leq b \leq 1$ and assume that $\alpha, \beta \in 2^\omega$ are balanced with $\varrho(\alpha) = a$ and $\varrho(\beta) = b$. For each $n$ let $x_n\in 2^{\omega}$ be defined as 
\begin{equation*} x_n = {s^0_n}^\frown {s^1_n} ^\frown s^2_n \dots, 
\end{equation*}  
where $s^k_n \in 2^{<\omega}$ is either a segment of $\alpha$ or a segment of $\beta$ for all $k\in \omega$. Suppose that $\length(s^{1}_{n}) \to \infty$ and $x_n\to x$ as $n \to \infty$. Then $\varrho(x)\in \{\alpha,\beta\}$. 
\end{lemma}
\begin{proof}
Since $\alpha$ and $\beta$ are balanced, it is easy to see that for any segments $s \sqsubseteq \alpha$ and $t \sqsubseteq \beta$ of length $n\geq 1$ we have 
\begin{equation}
\label{e:abdens}
\left| \varrho(\alpha)-  \varrho(s) \right|\leq \frac 1n  \quad \text{and} \quad	\left| \varrho(\beta)- \varrho(t)\right|\leq \frac 1n.
\end{equation}
For $N\in \omega$ let $\mathcal{I}(N)$ denote the set of discrete intervals $I$ satisfying $I\subset [N,\infty)$.  
We claim that it is enough to prove that there exists an $N\in \omega$ such that 
\begin{equation}
\label{e:allN}
\text{ either $x \restriction I \sqsubseteq \alpha$ for each $I \in \mathcal{I}(N)$, or $x \restriction I \sqsubseteq \beta$ for each $I  \in \mathcal{I}(N)$}.
\end{equation}
Suppose that \eqref{e:allN} holds, we may assume by symmetry that $x \restriction I \sqsubseteq \alpha$ for each $I \in \mathcal{I}(N)$. We want to show that $\varrho(x) =\varrho(\alpha)$. Let $y\in 2^{\omega}$ be defined by $y(n)=x(n+N)$ for each $n\in \omega$. As $\varrho(y)=\varrho(x)$, it is enough to prove that $\varrho(y)=\varrho(\alpha)$. By \eqref{e:allN} for any $n\geq 1$ we have $y\restriction n \sqsubseteq \alpha$, so \eqref{e:abdens} implies that 
\begin{equation*} 
\left| \varrho(\alpha)- \varrho(y\restriction n)\right|\leq \frac 1n,
\end{equation*} 
thus $\varrho(y)=\varrho(\alpha)$ follows. 
	
Therefore it remains to show \eqref{e:allN}. Assume to the contrary that \eqref{e:allN} fails for all $N$. Then one can find non-empty discrete intervals $I_1, I_2, I_3, I_4\subset \omega$ with $I_1 < I_2 < I_3 < I_4$ such that $x \restriction I_1 \not \sqsubseteq \alpha$, $x \restriction I_2 \not \sqsubseteq \beta$, $x \restriction I_3 \not\sqsubseteq \alpha$, and $x \restriction I_4 \not\sqsubseteq\beta$. Fix $k\in \omega$ large enough so that $I_1 \cup I_2 \cup I_3 \cup I_4 \subset [0, k)$ and then fix $n\in \omega$ large enough such that $x_n\restriction k = x \restriction k$ and $\length(s_n^1)>k$. Then clearly $x_n \restriction I_1 \sqsubseteq s_n^0$ and $x_n \restriction I_2 \sqsubseteq s_n^0$, or $x_n\restriction I_3 \sqsubseteq s_n^1$ and $x_n\restriction I_4 \sqsubseteq s_n^1$. Using that both $s_n^0$ and $s_n^1$ are segments of $\alpha$ or $\beta$, and the fact that $x$ and $x_n$ coincide on these intervals, we obtain a contradiction.
\end{proof}
	
\subsection{The proof of the Main Theorem} Finally, in this subsection we are ready to prove Theorem~\ref{t:characterization}. We need the following technical lemma, which is implicitly contained in \cite{microsets}.

\begin{lemma}
\label{l:building K}
Let $K_n \subset [0, 1]^d~(n\geq 1)$ be compact sets with $\dim_H E=\overline{\dim}_B \, E$ for all $E\in \iM(\{K_n\}_{n\geq 1})$ and let $\gamma=\sup\{ \dim_H K_n: n\geq 1\}$. Then there exists a compact set $K\subset [0,1]^d$ such that 
\begin{equation*}
\dim_H E=\overline{\dim}_B \, E \text{ for all } E\in \iM_K    
\end{equation*}
and 
\begin{equation*} \{\dim_H E : E \in \iM_K\} =\left\{\dim_H E : E \in  \iM(\{K_n\}_{n\geq 1})\right\} \cup \{\gamma\}.
\end{equation*} 
\end{lemma}
\begin{proof}
Apply the proof of \cite[Theorem~1.3]{microsets} to the multiset $\Omega^0=\{Q_i\}_{i\geq 1}$, where $\Omega^0$ is an enumeration of $\{K_n\}_{n\geq 1}$ such that each set $K_n$ is repeated infinitely often. 
\end{proof}
	
\begin{theorem}[Main Theorem]
\label{t:characterization}
Let $\dim$ be one of $\dim_H$, $\underline{\dim}_B$, or $\overline{\dim}_B$.
Let $d\geq 1$ and let $A \subset [0, d]$ be a non-empty set. Then the following are equivalent:
\begin{enumerate}
\item \label{i2} There exists a compact set $K \subset \R^d$ such that $\{\dim E : E \in \iM_K\} = A$;
\item \label{i3} $A$ is an analytic set which contains its infimum and supremum;
\item \label{i1} There exists a compact set $K \subset \R^d$ such that $\dim_H E=\overline{\dim}_B \, E$ for all $E\in \iM_K$ and $\{\dim_H E : E \in \iM_K\} = A$.
\end{enumerate} 
\end{theorem}

\begin{proof}  
The direction $\eqref{i1} \Rightarrow \eqref{i2}$ is straightforward by Fact~\ref{f:ineq}.

Now we prove $\eqref{i2} \Rightarrow \eqref{i3}$. Assume that $A = \{\dim E : E \in \iM_K\}$ for some compact set $K \in \iK(\R^d)$. Theorem~\ref{t:Fu} yields that $A$ contains its infimum and supremum as well. To see that $A$ is analytic, note that the set $\iM_K$ is $F_\sigma$, since the set $\{E \in \iM_K : E \cap [\eps, 1 - \eps]^d \neq \emptyset\}$ is closed for any $\varepsilon \in (0,1/2)$. Mattila and Mauldin  \cite{MM} proved that the mapping $\dim \colon \iK(\R^d)\to [0,d]$ is Borel measurable. Therefore, we obtain that $A$ is the image of a Borel set under a Borel map, hence it is analytic by \cite[Proposition~14.4]{Ke}. 
	
Finally, we show $\eqref{i3} \Rightarrow \eqref{i1}$. Fix an analytic set $A \subset [0, d]$ which contains its infimum and supremum. Define $B=\{z/d: z\in A\}$, then $B\subset [0,1]$ is analytic, and set $a = \min B$ and $b = \max B$. If $a = b = 0$, then $K$ can be a singleton. Hence we may assume that $b > 0$. Applying Lemma \ref{l:varphi existence} for the analytic set $B$ yields a map $\varphi \colon 2^{<\omega} \to [0,\infty)$ such that $\overline{\varphi}(x) = \lim_{n \to \infty} \varphi(x \restriction n)$ exists for all $x\in 2^{\omega}$ and $\overline{\varphi}(2^{\omega})=B$. We first check that we may assume 
\begin{equation}
\label{e:q3 a <= varphi <= b}
a \le \varphi(s) \le b \text{ for each $s \in 2^{<\omega}$.}
\end{equation}
Indeed, let us replace a value $\varphi(s)$ by $a$ if $\varphi(s) < a$, and replace $\varphi(s)$ by $b$ if $\varphi(s) > b$. As $\overline{\varphi}(x) \in B \subset [a, b]$ for each $x \in 2^\omega$, the values of $\overline{\varphi}$ do not change by modifying $\varphi$ in this way. 
	
We now construct a continuous map $\psi \colon 2^\omega \to 2^\omega$ and then use compact sets of the form $(K(\psi(x)))^d$ to construct $K$. Let $\alpha$ and $\beta$ be the sequences provided by Lemma~\ref{l:balanced} for $a$ and $b$. To construct $\psi$, first we specify a mapping $\phi \colon  2^{<\omega} \to 2^{<\omega}$ such that $\phi(s)$ is a segment of either $\alpha$ or $\beta$ for all $s\in 2^{<\omega}$. Let $\phi(\emptyset) = \emptyset$. For $s \in 2^{<\omega}$ with $\length(s)=n\geq 1$ define $\phi(s) = (\alpha \restriction n) ^\frown (\beta \restriction k)$, where $k=k(s)$ is a positive integer such that $ \sqrt{n}-1<k< n \sqrt{n}+1$ and 
\begin{equation} \label{e:q3  lambda close to varphi}
\left|\frac{na + k b}{n+k} - \varphi(s)\right| \le \frac{2}{\sqrt{n}},
\end{equation}
where the existence of such a $k$ can be checked as follows. Inequality \eqref{e:q3 a <= varphi <= b} implies that $\varphi(s) \in [a, b]$, and consider the real function \begin{equation*} 
r\colon [\sqrt{n}-1,n\sqrt{n}+1] \to [a,b], \quad  r(z)=\frac{an+bz}{n+z}.
\end{equation*}  
The inequalities 
\begin{equation*} 
\left|r\left(\sqrt{n}\right)-a\right|\leq \frac{2}{\sqrt{n}}, \quad  \left|r\left(n\sqrt{n}\right)-b \right| \leq \frac{1}{\sqrt{n}},\text { and }  r(z+1)-r(z)\leq \frac{1}{n}
\end{equation*} 
for all $\sqrt{n}-1\leq z\leq n\sqrt{n}$ easily imply the existence of $k$ satisfying \eqref{e:q3  lambda close to varphi}. 
	
For $x\in 2^{\omega}$ define 
\begin{equation} \label{e:psi}
\psi(x) = \phi(x\restriction 0) ^\frown \phi(x \restriction 1) ^\frown  \phi(x \restriction 2)   \dots.
\end{equation}
As $\phi(s)\neq \emptyset$ whenever $s\neq \emptyset$, the map $\psi$ is continuous. Since $b>0$ and $\beta$ is balanced with $\varrho(\beta)=b>0$, we obtain that $\psi(x)\neq \mathbf{0}$ for all $x\in 2^{\omega}$, where $\mathbf{0}\in 2^{\omega}$ is the zero sequence. We now claim that $\psi$ satisfies 
\begin{equation} \label{e:dvarphi}
\varrho(\psi(x)) = \overline{\varphi}(x)\text{ for each $x \in 2^\omega$.} 
\end{equation}
Let us fix $x \in 2^\omega$, and let 
\begin{equation*} 
\psi_n(x) = \phi(x \restriction 0) ^\frown \phi(x \restriction 1) ^\frown \dots ^\frown  \phi(x \restriction n).
\end{equation*} 
An elementary calculation using $k(s)=o(n^2)$ as $\length(s)=n\to \infty$ shows that 
\begin{equation*} 
\frac{\length(\phi(x \restriction n))}{\length(\psi_n(x))} \to 0 \text{ as } n \to \infty,
\end{equation*}  
hence it is enough to show that $\varrho(\psi_n(x)) \to \overline{\varphi}(x)$. Therefore, it is enough to show that $\varrho(\phi(x \restriction n)) \to \overline{\varphi}(x)$. Since $\alpha$ and $\beta$ are balanced, for $k=k(x \restriction n)$ we obtain 
\begin{equation} \label{e:nk}
|\varrho(\alpha \restriction n) - a| \le \frac{1}{n} \quad \text{and} \quad  
|\varrho(\beta \restriction k) - b| \le \frac{1}{k}.
\end{equation}  
Then \eqref{e:q3  lambda close to varphi}  and \eqref{e:nk} imply that 
\begin{align*} 
|\varrho(\phi(x \restriction n)) - \varphi(x \restriction n) |&= \left|\frac{n \varrho(\alpha \restriction n) +k \varrho(\beta \restriction k)}{n + k} - \varphi(x \restriction n)\right| \\
&\le \frac{2}{n + k} + \frac{2}{\sqrt{n}},
\end{align*}
which tends to $0$. This implies that $\varrho(\phi(x \restriction n)) \to \overline{\varphi}(x)$, so the proof of \eqref{e:dvarphi} is complete. As $\psi$ is continuous, Fact~\ref{f:cont} implies that the map
\begin{equation} \label{e:K_x}
x \mapsto K(\psi(x)) \text{ is also continuous.}
\end{equation}
Let $\{x_n\}_{n\geq 1}$ be a dense subset of $2^\omega$ such that
\begin{equation} \label{e:x1def} 
\overline{\varphi}(x_1) = b.
\end{equation} 
Define  
\begin{equation*} K_n = K(\psi(x_n))^d \text{ for all } n\geq 1.
\end{equation*} 
First we show that 
\begin{equation} \label{e:subs} A\subset \{\dim_H E : E \in \iM(\{K_n\}_{n\geq 1})\}.
\end{equation} 
Let $z \in A$ be arbitrary. By the definition of $\varphi$ there exists $x \in 2^\omega$ such that $\overline{\varphi}(x) = z/d\in B$. Fact~\ref{f:dim} and \eqref{e:dvarphi} imply
\begin{equation*}
\dim_H (K(\psi(x))^d)=d \varrho(\psi(x))=z.
\end{equation*} 
As $\{x_n\}_{n\geq 1}$ is dense in $2^{\omega}$, we can find a subsequence of positive integers $k_n\uparrow \infty$ such that $x_{k_n} \to x$. By \eqref{e:K_x} we obtain that 
\begin{equation*} 
K_{k_n}=K(\psi(x_{k_n}))^d \to K(\psi(x))^d \quad \text{as} \quad  n\to \infty.
\end{equation*} 
As $\psi(x)\neq \mathbf{0}$, we obtain $K(\psi(x))\cap (0,1)\neq \emptyset$, hence $K(\psi(x))^d\cap (0,1)^d\neq \emptyset$. Thus $K(\psi(x))^d \in \iM(\{K_n\}_{n\geq 1})$, so $z\in \{\dim_H E : E \in \iM(\{K_n\}_{n\geq 1})\}$ proving \eqref{e:subs}.

Now we claim 
\begin{equation} \label{e:EiG} 
\dim_H E=\overline{\dim}_B \, E\in A \text{ for all } 
E \in \iM(\{K_n\}_{n\geq 1}).
\end{equation} 
Let $E \in \iM(\{K_n\}_{n\geq 1})$ and let $\dim$ be one of $\dim_H$ or $\overline{\dim}_B$, we will calculate $\dim E$ independently of the choice of the dimension. Since we are only interested in $\dim E$, by Theorem~\ref{t:Cx} we may suppose that $E=K(y)^d$ for some $y\in 2^{\omega}$ for which there exists a subsequence of positive integers $k_n\uparrow \infty$ and $m_n\in \omega$ such that $T^{m_n}(\psi(x_{k_n}))\to y$. We may assume by choosing a subsequence that $x_{k_n}\to x$ for some $x \in 2^\omega$. 

First suppose that $\{m_n\}_{n\geq 1}$ is bounded. By choosing a subsequence we may assume that $m_n=m$ for all $n$. The continuity of $\psi$ implies $y=T^m(\psi(x))$. As $\varrho(T^m(\psi(x)))=\varrho(\psi(x))$, using \eqref{e:dvarphi} and $\overline{\varphi}(x)\in B$ we obtain  
\begin{equation*} \dim E=\dim (K(y)^d)=\dim (K(\psi(x))^d)=d\varrho(\psi(x))=d\overline{\varphi}(x)\in A,
\end{equation*} 
and we are done. 

Next, assume that $\{m_n\}_{n\geq 1}$ is not bounded. We may suppose by choosing a subsequence that $m_n\to \infty$. Applying Lemma \ref{l:balanced} for $T^{m_n}(\psi(x_{k_n}))\in 2^{\omega}$ implies that $\varrho(y)=a$ or $\varrho(y)=b$, so 
\begin{equation*}
\dim E=\dim (K(y)^d)=d\varrho(y)\in \{\min A, \max A\}, 
\end{equation*} 
which finishes the proof of \eqref{e:EiG}. 

Finally, we are able to define $K$ and finish the proof of \eqref{i1}.
Since \eqref{e:EiG} yields $\dim_H E=\overline{\dim}_B \, E$ for all $E \in \iM(\{K_n\}_{n\geq 1})$, applying Lemma~\ref{l:building K} for $K_n$ implies that there exists a compact set $K\subset [0,1]^d$ such that 
\begin{equation} \label{e:dimeq}
\dim_H E=\overline{\dim}_B \, E \text{ for all } E\in \iM_K    
\end{equation}
and 
\begin{equation} \label{e:gam} \{\dim_H E : E \in \iM_K\} =\left\{\dim_H E : E \in  \iM(\{K_n\}_{n\geq 1})\right\} \cup \{\gamma\}, 
\end{equation} 
where $\gamma \defeq \sup\{\dim_H K_n: n\geq 1\}$. In order to remove $\gamma$ from the right hand side of \eqref{e:gam} we prove
\begin{equation} \label{e:gamk1}
\gamma \in \left\{\dim_H E : E \in  \iM(\{K_n\}_{n\geq 1})\right\}.
\end{equation} 
Claim~\ref{cl:1}, \eqref{e:dvarphi}, and \eqref{e:x1def} imply that 
\begin{equation} \label{e:=b} 
\dim_H K_1=\varrho(\psi(x_1)) = \overline{\varphi}(x_1) = b, 
\end{equation}
and also adding \eqref{e:q3 a <= varphi <= b} yields
\begin{equation*} 
\gamma=\sup\{ \dim_H K_n: n\geq 1\}=\sup\{\varrho(\psi(x_n)): n\geq 1\}=\sup\{ \overline{\varphi}(x_n) : n\geq 1\}\leq b. 
\end{equation*}
This and \eqref{e:=b} imply $\gamma= \dim_H K_1$, and clearly $K_1\in \iM(\{K_n\}_{n\geq 1})$, thus \eqref{e:gamk1} follows. By \eqref{e:gam}, \eqref{e:gamk1}, \eqref{e:subs}, and \eqref{e:EiG} we obtain 
\begin{equation} \label{e:MKn}
\{\dim_H E : E \in \iM_K\} =\left\{\dim_H E : E \in  \iM(\{K_n\}_{n\geq 1})\right\}=A.
\end{equation} 
Equations \eqref{e:dimeq} and \eqref{e:MKn} imply \eqref{i1}, and the proof of the theorem is complete.
\end{proof}

\section{Compact families of compact sets} \label{s:K}

\subsection{The case of Hausdorff dimension} \label{ss:Haus}
The goal of this subsection is to prove Theorem~\ref{t:compact family} after some preparation. We define parametrized fractal percolations in axis parallel cubes $Q\subset \R^d$ as follows.

\begin{definition} \label{d:p}
Let $Q\subset \R^d$ be an axis parallel cube with side length $r$. For all $n\ge 0$, $Q$ can be written as a union of $2^{dn}$ many non-overlapping closed cubes of side length $r2^{-n}$. We denote by $\iD_n$ the collection of these subcubes, and the elements of $\bigcup_{n \ge 0} \iD_n$ are called the \emph{dyadic subcubes of $Q$}. Given $\alpha_n\in [0,d]$ for all $n\geq 1$, we construct a random compact set $\Gamma(\{\alpha_n\}_{n\geq 1})\subset Q$ as follows. We keep each of the $2^d$ cubes in $\iD_1$ with probability $2^{-\alpha_1}$. Let $\Delta_1\subset \iD_1$ be the collection of kept cubes and let $S_1=\bigcup \Delta_1$ be their union. If $\Delta_{n-1}\subset \iD_{n-1}$ and $S_{n-1}=\bigcup \Delta_{n-1}$ are already defined, then we keep each cube in $D\in \iD_{n}$ for which $D\subset S_{n-1}$ independently with probability $2^{-\alpha_{n}}$. Denote by $\Delta_{n}\subset \iD_{n}$ the collection of kept cubes and by $S_{n}=\bigcup \Delta_{n}$ their union. Define our \emph{percolation limit set with generation dependent retention probabilities $2^{-\alpha_n}$} as
\begin{equation*} 
\Gamma(\{\alpha_n\}_{n\geq 1})=\bigcap_{n=1}^{\infty} S_n.
\end{equation*}
If $\alpha_n=\alpha$ for all $n\geq 1$ then we simply use the notation $\Gamma(\alpha)$ instead of $\Gamma(\{\alpha_n\}_{n\geq 1})$.

We say that a random set $X$ \emph{stochastically dominates} $Y$ if they can be coupled (defined on a common probability space) such that $Y\subset X$ almost surely. For two sets $A,B$ we write $A\subset^{\star} B$ if $A\setminus B$ is countable. 
\end{definition}

The following theorem is due to Hawkes \cite[Theorem~6]{H} in the context of trees, see also \cite[Theorem~9.5]{MP}.

\begin{theorem}[Hawkes] \label{t:H} For every $\beta\in [0,d]$ and every compact set $K\subset Q$ the following properties hold:
\begin{enumerate}
\item \label{e:H1} if $\dim_H K<\beta$, then almost surely, $K\cap \Gamma(\beta)=\emptyset$,
\item \label{e:H2} if $\dim_H K>\beta$, then $K\cap \Gamma(\beta)\neq \emptyset$ with positive probability,
\item \label{e:H3} if $\dim_H K>\beta$, then almost surely, $\dim_H(K\cap \Gamma(\beta))\leq \dim_H K-\beta$.
\end{enumerate} 
\end{theorem}

\begin{lemma} \label{l:H}
Let $K\subset Q$ be compact and let $0<\beta <\dim_H K$. Then there exists a constant $c=c(Q,K,\beta)>0$ such that the following holds. If $\alpha_n \in [0,\beta]$ for all $n\geq 1$ and $\alpha_n\to \alpha $, then 
\begin{equation} \label{e:Kc}
\P(\dim_H ( K\cap \Gamma(\{ \alpha_n \}_{n\geq 1}))\geq \beta-\alpha)\geq c.
\end{equation}
\end{lemma}
\begin{proof} Define $c=\P(K\cap \Gamma(\beta)\neq \emptyset)$, we have $c>0$ by Theorem~\ref{t:H}~\eqref{e:H2}. For all $n\geq 1$ let $\delta_n=\beta-\alpha_n\geq 0$  and consider $K\cap \Gamma(\{\alpha_n\}_{n\geq 1})\cap \Gamma(\{\delta_n\}_{n\geq 1})$, where $\Gamma(\{\alpha_n\}_{n\geq 1})$ and $\Gamma(\{\delta_n\}_{n\geq 1})$ are independent. Since  $\Gamma(\{\alpha_n\}_{n\geq 1})\cap \Gamma(\{\delta_n\}_{n\geq 1})$ stochastically dominates $\Gamma(\beta)$, we obtain
	\begin{equation} \label{e:gamma1} 
	\P(K\cap \Gamma(\{\alpha_n\}_{n\geq 1})\cap \Gamma(\{\delta_n\}_{n\geq 1})\neq \emptyset)\geq c.
	\end{equation} 
Fix any compact set $C\subset \R^d$ with $\dim_H C<\beta-\alpha$. Let us choose $\delta>0$ such that $\dim_H C<\delta<\beta-\alpha$. Applying Theorem~\ref{t:H}~\eqref{e:H1} yields $\P(C\cap \Gamma(\delta)\neq \emptyset) = 0$. As $\Gamma(\{\delta_n\}_{n\geq 1})$ is stochastically dominated by the union of finitely many similar copies of $\Gamma(\delta)$, we obtain that $\P(C \cap  \Gamma(\{\delta_n\}_{n\geq 1})\neq \emptyset) = 0$.
The independence of $\Gamma(\{\alpha_n\}_{n\geq 1})$ and $\Gamma(\{\delta_n\}_{n\geq 1})$, and Fubini's theorem allow us to replace $C$ by $K\cap \Gamma(\{\alpha_n\}_{n\geq 1})$, and we obtain
\begin{equation*} 
\P\Big(K\cap \Gamma(\{\alpha_n\}_{n\geq 1})\cap \Gamma(\{\delta_n\}_{n\geq 1})\neq \emptyset \text{ and } \dim_H(K\cap \Gamma(\{\alpha_n\}_{n\geq 1}))<\beta-\alpha\Big)=0.
	\end{equation*}
This and \eqref{e:gamma1} imply \eqref{e:Kc}, and the proof is complete.
\end{proof}

\begin{theorem}
\label{t:compact family}
Let $K \subset \R^d$ be a non-empty compact set and let $A \subset [0, \dim_H K]$. Then the following are equivalent: 
\begin{enumerate}
\item \label{ic1} There is a compact set $\mathcal{C} \subset \mathcal{K}(K)$ with $\{\dim_H C : C \in \mathcal{C}\} = A$;
\item \label{ic2} $A$ is an analytic set. 
\end{enumerate} 
\end{theorem}

\begin{proof}  
First we prove $\eqref{ic1} \Rightarrow \eqref{ic2}$. Since $\dim_H \colon \iK(\R^d)\to [0,d]$ is Borel measurable by \cite[Theorem 2.1]{MM}, if $\{\dim_H C : C \in \mathcal{C}\} = A$ then $A$ must be analytic as the image of a compact set under a Borel map, see e.g.~\cite[Proposition~14.4]{Ke}. 	
	
Now we show $\eqref{ic2} \Rightarrow \eqref{ic1}$. Let $\gamma=\dim_H K$. By Fact~\ref{f:dimK} we can fix $z_0 \in K$ such that $\dim_H B(z_0,r)= \gamma$ for all $r>0$. We may assume that $A \subset(0, \gamma]$ is non-empty and analytic, since if $A=\emptyset$ then $\iC=\emptyset$ works, and if $0\in A$ and an appropriate family $\mathcal{C}$ for the analytic set $A \setminus \{0\}$ is constructed, then $\mathcal{C}\cup\{\{ z_0 \}\}$ works for $A$.  
	
Fix a sequence of positive numbers $\beta_k \uparrow \gamma$, and for all $k\geq 1$ let $Q_k$ be a cube around $z_0$ of side length $1/k$. Fix $k\geq 1$ arbitrarily, and let $c_k>0$ be the constant we obtain by applying Lemma~\ref{l:H} for $K\cap Q_k\subset Q_k$ and $\beta_k$. Choose $i_k\in \N^+$ large enough so that 
\begin{equation}
\label{e:choice of i_k}
(1-c_k)^{i_k} < \frac 12.
\end{equation}
We will run $i_k$ independent, parameterized families of percolations inside $Q_k$. For all $n\geq 0$ let $\iD_n^k$ denote the collection of dyadic subcubes of $Q_k$ with side length $(1/k)2^{-n}$ and let $\iD^k=\bigcup_{n=1}^{\infty} \iD_n^k$. For any $D\in \iD^k$ let $u^k_i(D)$ be a random variable uniformly distributed in $[0, 1]$ such that the family $\{u^k_i(D): k \ge 1, \, i \leq  i_k, \, D \in \iD^k\}$ is independent. Assume that a sequence $\{\alpha_n\}_{n\geq 1}$ is given such that $\alpha_n\in [0,\gamma)$ for all $n\geq 1$ and $\alpha_n \to \alpha \in [0,\gamma)$. We define $\Gamma_i^k(\{ \alpha_n \}_{n\geq 1})$ as follows. Let $S_0=Q_k$ and $\iD^k_0=\{Q_k\}$, and for each $1\leq i\leq i_k$ let 
\begin{align*} 
&\Delta_1=\Delta_{i,1}^k(\{\alpha_n\}_{n\geq 1})=\{D\in \iD_1^k: u_i^k(D)\leq 2^{-\alpha_1} \}, \\
&S_1=S_{i,1}^k(\{\alpha_n\}_{n\geq 1})=\bigcup \Delta_1.
\end{align*} 
If $\Delta_{m-1}$ and $S_{m-1}$ are already defined, let 
\begin{align} \label{a:Sm}
\begin{split}
&\Delta_{m}=\Delta_{i,m}^k(\{\alpha_n\}_{n\geq 1})=\{D\in \iD_m^k: D\subset S_{m-1} \text{ and } u_i^k(D)\leq 2^{-\alpha_m} \}, \\
&S_m=S_{i,m}^k(\{\alpha_n\}_{n\geq 1})=\bigcup \Delta_m.
\end{split}
\end{align} 
Finally, we define 
\begin{equation} \label{e:Gki} 
\Gamma^k_i(\{\alpha_n\}_{n \geq 1}) = \bigcap_{m=0}^{\infty} S_m.
\end{equation} 
Since for any $\delta < \alpha$ the percolation $\Gamma^k_i(\{\alpha_n\}_{n \geq 1})$ is stochastically dominated by the union of finitely many similar copies of $\Gamma(\delta)$, by Theorem~\ref{t:H}~\eqref{e:H3} for all $1\leq i\leq i_k$ we obtain
\begin{equation}
\label{e:percolation probability}
\P\left(\dim_H \left(K \cap \Gamma^k_i(\{\alpha_n\}_{n \geq 1})\right) 
\le \gamma- \alpha\right) = 1.
\end{equation}
Moreover, if $\alpha_n \in [0,\beta_k]$ for all $n\geq 1$, then Lemma~\ref{l:H}, \eqref{e:choice of i_k}, and the independence of the processes $\Gamma^k_i(\{\alpha_n\}_{n \geq 1})$ for $1\leq i\leq i_k$ imply 
\begin{equation}
\label{e:percolation probability 2}
\P\left(\dim_H \left(K \cap \left(\bigcup_{i=1 }^{i_k}\Gamma^k_i(\{\alpha_n\}_{n \geq 1})\right)\right) \ge \beta_k - \alpha \right) \geq \frac 12.
\end{equation}
For each $k \ge 1$ for any $D \in \mathcal{D}^k$ satisfying $D \cap K \neq \emptyset$ we choose a point $z_D \in D \cap K$. For all 
$n\geq 0$ let 
\begin{equation*} 
\iC^k_n=\{D\in \iD^k_n: D\cap K\neq \emptyset\},
\end{equation*}	
and define the countable random set $F^k_i=F^k_i(\{\alpha_n\}_{n \geq 1})$ as
\begin{equation} \label{e:Fki}
F^k_i= \bigcup_{n=0}^{\infty} \big\{z_D : D \in \mathcal{C}^k_n, \, D \subset S_n, \, \nexists C\in \iC^k_{n+1}  \text{ with } C\subset S_{n+1} \cap D \big\}.
\end{equation}
We now claim that 
\begin{equation}\label{e:set is compact}
F^k_i(\{\alpha_n\}_{n \geq 1}) \cup \Gamma^k_i(\{\alpha_n\}_{n \geq 1}) \text{ is compact for each $k \ge 1$ and $1\leq i \leq  i_k$}.
\end{equation} 
Indeed, $F^k_i(\{\alpha_n\}_{n \geq 1}) \setminus S_m$ is finite for all $m\geq 1$, hence $S_m^* = S_m \cup F^k_i(\{\alpha_n\}_{n \geq 1})$ is compact. Therefore $F^k_i(\{\alpha_n\}_{n \geq 1}) \cup \Gamma^k_i(\{\alpha_n\}_{n \geq 1})  = \bigcap_{m=1}^{\infty} S_m^*$ is compact as well, which completes the proof of \eqref{e:set is compact}. 
	
For all $1\leq i\leq i_k$ let 
\begin{equation} \label{e:gamma*} 
\Gamma^{*}_{k,i}(\{\alpha_n\}_{n \geq 1})=F^k_i(\{\alpha_n\}_{n \geq 1}) \cup \Gamma^k_i(\{\alpha_n\}_{n \geq 1}), 
\end{equation}
and define
\begin{equation} \label{e:Gammadef} 
\Gamma^*(\{\alpha_n\}_{n \geq 1}) = \{z_0 \} \cup \bigcup_{k=1}^{\infty} \bigcup_{i=1}^{i_k} \Gamma^{*}_{k,i}(\{\alpha_n\}_{n \geq 1}).
\end{equation} 
Using \eqref{e:set is compact} and the fact that $\Gamma^{*}_{k,i}(\{\alpha_n\}_{n \geq 1}) \subset Q_k$ and $Q_k\to \{z_0\}$, it is clear that $\Gamma^*(\{\alpha_n\}_{n \geq 1})$ is compact. As $\alpha<\gamma$ and $\alpha_n<\gamma$ for all $n\geq 1$,  we have $\sup \{\alpha_n: n\geq 1\}\leq \beta_k$ for all large enough $k$, so we can apply \eqref{e:percolation probability 2} for large values of $k$. Therefore  \eqref{e:percolation probability}, \eqref{e:percolation probability 2}, and the independence of the processes defining each $\Gamma^k_i(\{\alpha_n\}_{n \geq 1})$ yield
\begin{equation*}
\P\left(\dim_H \left(K \cap \Gamma^*(\{\alpha_n\}_{n \geq 1})\right) = \gamma - \alpha\right) = 1.
\end{equation*}
Our coupling of percolations clearly implies the following monotonicity: Almost surely, for all sequences $\{\alpha_n\}_{n \geq 1}$ and $\{ \alpha'_n\}_{n\geq 1}$ we have
	\begin{equation}
	\label{e:monotonicity}
	\Gamma^*(\{\alpha_n\}_{n \geq 1}) \subset^{\star} \Gamma^*(\{\alpha'_n\}_{n\geq 1}) \text{ if $\alpha_n \ge \alpha'_n$ for each $n$}.
	\end{equation}
Let $Q =\Q \cap [0, \gamma)$, and define the set 
\begin{equation*} 
Q^* = \{\{\alpha_n\}_{n \geq 1} : \alpha_n\in Q \text{ for all $n$ and $\alpha_n$ is eventually constant}\}.
\end{equation*}  
Clearly, $Q^*$ is countable. Therefore, almost surely we have 
\begin{equation}
\label{e:dimension prob for Q}
\dim_H (K \cap \Gamma^*(\{\alpha_n\}_{n \geq 1})) = \gamma - \alpha  \text{ for all $\{\alpha_n\}_{n \geq 1} \in Q^*$ with $\alpha_n \to \alpha$}.
\end{equation}
	
Now we are ready to define our family of compact sets $\mathcal{C}$. By Lemma~\ref{l:varphi existence} there exists a map $\varphi \colon 2^{<\omega} \to [0,\infty)$ such that $\overline{\varphi}(x) = \lim_{n \to \infty} \varphi(x \restriction n)$ exists for all $x \in 2^\omega$ and $\overline{\varphi}(2^\omega) = A$. Since $A \subset (0, \gamma]$, we may assume that 
\begin{equation} \label{e:0gamma}
0<\varphi(s)\leq \gamma  \text{ for all } s\in 2^{<\omega}:
\end{equation} 
Indeed, otherwise for all $s\in 2^n$ we can replace $\varphi(s)$ by $\gamma 2^{-n}$ if $\varphi(s)\leq 0$ and by $\gamma$ if $\varphi(s)>\gamma$, which modifications do not change the limit $\overline{\varphi}$. For each $x \in 2^\omega$ let $\alpha(x)=\{\alpha(x)_n\}_{n\geq 1}$ such that
\begin{equation*} 
\alpha(x)_n=\gamma-\varphi(x\restriction n) \text{ for all } n\geq 1. 
\end{equation*} 
Let us define $\mathcal{C}$ as 
\begin{equation*} 
\mathcal{C} = \{K \cap \Gamma^*(\alpha(x)) : x \in 2^\omega\}.
\end{equation*}
It is clear that $\mathcal{C}$ is a random family of compact sets. Now we prove that, almost surely, $\{\dim_H C : C \in \mathcal{C}\} = A$. Assume that the event of \eqref{e:dimension prob for Q} holds, it is enough to show that $\dim_H (K \cap \Gamma^*(\alpha(x))) = \overline{\varphi}(x)$ for all $x \in 2^\omega$. Let $x \in 2^\omega$ be fixed. The definition of $\varphi$ and \eqref{e:0gamma} imply that $\alpha(x)_n\in [0,\gamma)$ for all $n$, and $\alpha(x)_n$ converges to $\alpha_x\defeq \gamma-\overline{\varphi}(x)\in  [0,\gamma)$. Hence for any $\varepsilon > 0$ we can find sequences $\{\alpha'_n\}_{n \geq 1}, \{\alpha''_n\}_{n \geq 1} \in Q^*$ such that 
\begin{equation*} 
\alpha'_n \le \alpha(x)_n \le \alpha''_n \text{ and } \alpha''_n-\alpha'_n\leq \eps \text{ for each } n.
\end{equation*} 
Then $\alpha'_n\to \alpha'$ and $\alpha''_n\to \alpha''$ such that 
\begin{equation} \label{e:alpha'} 
\alpha'\leq \alpha_x\leq \alpha'' \text{ and } \alpha''-\alpha'\leq \eps.
\end{equation} 
By \eqref{e:dimension prob for Q} we have 
\begin{equation} \label{e:dimg}
\dim_H (K \cap \Gamma^*(\{\alpha'_n\}_{n \geq 1})) = \gamma - \alpha' \text{ and } \dim_H(K \cap \Gamma^*(\{\alpha''_n\}_{n \geq 1})) = \gamma - \alpha''.
\end{equation}
Monotonicity \eqref{e:monotonicity} yields 
\begin{equation} \label{e:mon}
\Gamma^{*} (\{\alpha''_n\}_{n\geq 1} )\subset^{\star} \Gamma^{*}(\alpha(x)) \subset^{\star} 
\Gamma^{*} (\{\alpha'_n\}_{n\geq 1}).
\end{equation}
As $\eps>0$ was arbitrary, \eqref{e:alpha'}, \eqref{e:dimg}, and \eqref{e:mon} imply that 
\begin{equation*}
\dim_H(K \cap \Gamma^*(\alpha(x))) =\gamma-\alpha_x=\overline{\varphi}(x),
\end{equation*}
which proves that $\{\dim_H C : C \in \mathcal{C}\} = A$ almost surely.

Finally, we check that $\mathcal{C}$ is compact with probability $1$. It is enough to show that the map $x\to K \cap \Gamma^*(\alpha(x))$ is continuous almost surely, then $\iC$ will be compact as a continuous image of the compact set $2^\omega$. Since $Q_k\to \{z_0\}$ and $\Gamma^{*}_{k,i}(\{\alpha_n\}_{n \geq 1}) \subset Q_k$ for all $k$ and $i$, by \eqref{e:Gammadef} it is enough to prove that for arbitrarily given $k\geq 1$ and $i\in \{1,\dots,i_k\}$ the map $x \mapsto C_x\defeq K\cap \Gamma_{k,i}^*(\alpha(x))$ is continuous. Assume that $x,y\in 2^{\omega}$ are given such that $x \restriction m= y \restriction m$ for some $m\geq 1$. It is enough to show 
\begin{equation} \label{e:dH} 
d_H(C_x, C_y)\leq 2^{-m}\diam Q_k.
\end{equation} 
By \eqref{a:Sm} we obtain that $S^k_{i,m}(\alpha(x))=S^k_{i,m}(\alpha(y))\defeq S_m$. 
The definition \eqref{e:Gki} yields $\Gamma^k_i(\alpha(x))\setminus S_m=\Gamma^k_i(\alpha(y))\setminus S_m=\emptyset$, and \eqref{e:Fki} implies 
$F^k_i(\alpha(x))\setminus S_m=F^k_i(\alpha(y))\setminus S_m$. Therefore \eqref{e:gamma*} yields $C_x \setminus S_m=C_y\setminus S_m$, and we obtain that     
\begin{equation} \label{e:dH1}
d_H(C_x, C_y)\le d_H(C_x\cap S_m, C_y\cap S_m).
\end{equation} 
Let $D\in \iD^k_m$ be arbitrary such that $D\subset S_m$. Definitions \eqref{e:Fki} and \eqref{e:gamma*} imply that $C_x\cap D\neq \emptyset$ iff $D\cap K\neq \emptyset$ iff $C_y\cap D\neq \emptyset$, so $\diam D=2^{-m}\diam Q_k$ yields 
\begin{equation} \label{e:dH2}
d_H(C_x\cap S_m, C_y\cap S_m)\leq 2^{-m}\diam Q_k.
\end{equation}
Then \eqref{e:dH1} and \eqref{e:dH2} imply \eqref{e:dH}, and the proof is complete. 
\end{proof}

\subsection{Box and packing dimensions} \label{ss:box}
In this subsection we prove Theorems~\ref{t:compbox} and \ref{t:comppack}. For the following equivalent version of the upper box dimension and for more alternative definitions see \cite[Chapter~3]{Fa}.

\begin{definition}
Let $(X,\rho)$ be a totally bounded metric space. We say that $S\subset X$ is a \emph{$\delta$-packing} if $\rho(x,y)>\delta$ for all distinct $x,y\in S$. Let $P_n(X)$ be the maximal cardinality of the $2^{-n}$-packings in $X$. 
\end{definition}

\begin{fact} \label{f:equiv} 
Let $X$ be a totally bounded metric space. Then for every $n\geq 0$ we have $N_n(X)\leq P_n(X)\leq N_{n+1}(X)$, so \begin{equation*}
\overline{\dim}_B \, X=\limsup_{n \to \infty} \frac{\log P_n(X)}{n\log 2}.
\end{equation*} 
\end{fact}

\begin{theorem} \label{t:compbox}
Let $K$ be a non-empty compact metric space and $A \subset [0, \overline{\dim}_B \, K]$. Then the following are equivalent: 
\begin{enumerate}
\item \label{i01} There is a compact set $\mathcal{C} \subset \mathcal{K}(K)$ with $\{\overline{\dim}_B \, C : C \in \mathcal{C}\} = A$;
\item \label{i02} $A$ is an analytic set. 
\end{enumerate} 
\end{theorem}
\begin{proof} First we show $\eqref{i01} \Rightarrow \eqref{i02}$. As $\overline{\dim}_B \, \colon \iK(\R^d)\to [0,d]$ is Borel measurable by \cite[Lemma~3.1]{MM}, if $\{ \overline{\dim}_B \, C : C \in \mathcal{C}\} = A$ then $A$ must be analytic as it is the image of a compact set under a Borel map, see e.g.~\cite[Proposition~14.4]{Ke}. 
	
Now we prove $\eqref{i02} \Rightarrow \eqref{i01}$. We may assume that $A\neq \emptyset$, otherwise $\iC=\emptyset$ works. If $\overline{\dim}_B \, K=0$ then $A=\{z\}$ works for any $z\in K$, so we may suppose $\overline{\dim}_B \, K>0$. Let $\alpha=\overline{\dim}_B \, K$ and choose a sequence $\alpha_n \uparrow \alpha$. By Fact~\ref{f:dimK} we can fix $z_0\in K$ such that 
\begin{equation*}
	\overline{\dim}_B \, B(z_0,r)=\alpha \text{ for all } r>0.
	\end{equation*}
Choose a sequence $k_n\uparrow \infty$ with $k_0=0$ such that for all $n\geq 1$ we have
\begin{equation} \label{e:alpha}  P_{k_{n}}(B(z_0,2^{-k_{n-1}}))\geq 2^{\alpha_{n} k_{n}}. 
\end{equation}
By Lemma~\ref{l:varphi existence} there is a map $\varphi \colon 2^{<\omega} \to [0, \infty)$ such that
\begin{equation*} 
\overline{\varphi}(x) = \lim_{n \to \infty} \varphi(x \restriction n) 
\end{equation*} 
exists for each $x \in 2^\omega$, and $\overline{\varphi}(2^\omega) = A$. We may assume that $\varphi(s)\leq \alpha_n$ for all $n\geq 0$ and $s\in 2^n$, otherwise we can replace $\varphi(s)$ by $\alpha_n$ without changing $\overline{\varphi}$. Let $C(\emptyset)=\{z_0\}$, and fix $n\geq 1$ and $s\in 2^n$ arbitrarily. We will define a finite set $C(s)$ close to $z_0$. Let $\ell(s)$ be the minimal integer such that $k_{n-1} \leq \ell(s) \leq k_{n}$ and  
\begin{equation} \label{e:ells}
P_{\ell(s)}(B(z_0,2^{-k_{n-1}}))\geq 2^{\varphi(s)\ell(s)},
\end{equation} 
by \eqref{e:alpha} the number $\ell(s)$ is well-defined. Choose a $2^{-\ell(s)}$-packing $C(s)$ of size $\lfloor 2^{\varphi(s)\ell(s)} \rfloor$ in $B(z_0,2^{-k_{n-1}})$. For all $x\in 2^{\omega}$ set
\begin{equation*} 
C(x)=\bigcup_{n=0}^{\infty} C(x\restriction n).
\end{equation*}
Clearly, $C(x)$ is a countable compact set with the only possible accumulation point $z_0$. If $x\restriction n=y\restriction n$, then $C(x)\setminus B(z_0,2^{-k_n})=C(y)\setminus B(z_0,2^{-k_n})$. Hence
\begin{equation*} 	
d_H(C(x),C(y))\leq 2^{-k_n},
\end{equation*} 
so the map $x\mapsto C(x)$ is continuous. Define
\begin{equation*} \iC=\{C(x): x\in 2^{\omega}\},
\end{equation*}
then $\iC\subset \iK(K)$ is compact as a continuous image of the compact set $2^{\omega}$. 

In order to prove $\{\overline{\dim}_B \, C: C\in \iC\}=A$, let $x\in 2^{\omega}$ be arbitrarily fixed. It is enough to show that 
\begin{equation} \label{e:BC}
\overline{\dim}_B \, C(x)=\overline{\varphi}(x).
\end{equation} 
Clearly, $C(x)$ contains the $2^{-\ell(x\restriction n)}$-packing $C(x\restriction n)$ of size $\lfloor 2^{\varphi(x\restriction n) \ell(x\restriction n)}\rfloor$ for all $n$, so Fact~\ref{f:equiv} yields that 
\begin{equation*} 	
\overline{\dim}_B \, C(x)\geq \limsup_{n\to \infty} \varphi(x\restriction n)= \overline{\varphi}(x).
\end{equation*}	
Let $\gamma>\overline{\varphi}(x)$ be arbitrary, in order to prove \eqref{e:BC} we only need to show that
\begin{equation} \label{e:upper}
\overline{\dim}_B \, C(x)\leq \gamma.
\end{equation}
For all large enough $n$ the construction, $\ell(x\restriction i) \leq k_i$, and elementary calculation yield
 \begin{equation} \label{e:sharp}
 \#\left(\bigcup_{i=0}^{n-1} C(x\restriction i)\right)\leq \sum_{i=0}^{n-1} 2^{\varphi(x\restriction i) k_i}\leq 2^{\gamma k_{n-1}}.
 \end{equation} 
For all $n$ and $\ell\in \{k_{n-1},\dots,k_n\}$ we claim
\begin{equation} \label{e:Nl} 
N_{\ell}(C(x\restriction n))\leq P_{\ell}(C(x\restriction n))\leq  2^{\varphi(x\restriction n)\ell}.
\end{equation}
The first inequality follows from Fact~\ref{f:equiv}. If $\ell<\ell(x\restriction n)$, then the second inequality follows from the minimality of $\ell(x\restriction n)$, see \eqref{e:ells}. If $\ell\geq \ell(x\restriction n)$, then \begin{equation*} 
P_{\ell}(C(x\restriction n))\leq \#C(x\restriction n)\leq 2^{\varphi(x\restriction n)\ell(x\restriction n)}\leq 2^{\varphi(x\restriction n)\ell}, 
\end{equation*}
so \eqref{e:Nl} holds. 

Now, inequalities \eqref{e:sharp}, \eqref{e:Nl}, and $\bigcup_{i=n+1}^{\infty} C(x\restriction i)\subset B(z_0,2^{-k_n})$ yield that for all large enough $n$ and for each $\ell\in \{k_{n-1},\dots,k_n\}$ 
we have
\begin{equation*}
N_{\ell}(C(x)) \leq 2^{\gamma k_{n-1}}+2^{\varphi(x\restriction n)\ell}+1\leq 2^{\gamma \ell+1}, 
	\end{equation*} 
so \eqref{e:upper} holds. The proof of the theorem is complete.
\end{proof}

For the packing dimension we prove the following theorem.

\begin{theorem} \label{t:comppack}
	Let $K$ be a non-empty compact metric space and let $A \subset [0, \dim_P K]$ be analytic. Then there is a compact set $\mathcal{C} \subset \mathcal{K}(K)$ with $\{\dim_P C : C \in \mathcal{C}\} = A$.
\end{theorem}

\begin{proof}  First suppose that $A\subset [0,\alpha]$ for some $\alpha<\dim_P K$. We may assume $A\neq \emptyset$, otherwise $\iC=\emptyset$ works. By  Lemma~\ref{l:packing}~\eqref{i:ii} we may suppose by shrinking $K$ if necessary that $\overline{\dim}_B \,  U\geq \dim_P U>\alpha$ for any non-empty open set $U\subset K$. By Lemma~\ref{l:varphi existence} there exists a map $\varphi \colon 2^{<\omega} \to [0, \infty)$ such that
\begin{equation*} \overline{\varphi}(x) = \lim_{n \to \infty} \varphi(x \restriction n)
\end{equation*} 
exists for each $x \in 2^\omega$, and $\overline{\varphi}(2^\omega) = A$. We may assume that $\varphi(s)\leq \alpha$ for all $s\in 2^{<\omega}$, otherwise we can replace $\varphi(s)$ by $\alpha$ without changing $\overline{\varphi}$. Let $c_n=2^{-n}$ for all $n\geq 1$. By compactness and Fact~\ref{f:equiv} we can choose a sequence $k_n\uparrow \infty$ such that $k_0=0$ and for all $z\in K$ and $n\geq 0$ there exists a positive integer $\ell=\ell(n,z)$ such that $k_{n}\leq \ell \leq c_{n+1} k_{n+1}\leq k_{n+1}-2$ and
\begin{equation}  \label{e:Pj} 
P_{\ell}(B(z,2^{-k_{n}-1}))\geq 2^{\alpha \ell}. 
\end{equation}
Let $m(\emptyset)=1$, $z_1(\emptyset)\in K$ be arbitrary, and $C(\emptyset)=B(z_1(\emptyset),1)$. Assume that $s\in 2^{n}$ for some $n\geq 0$, a positive integer $m(s)$, points $z_i(s)\in K$ for $1 \le i \le m(s)$ are given such that the distance of $B(z_i(s), 2^{-k_n})$ and $B(z_{i'}(s), 2^{-k_n})$ is greater than $2^{-k_n}$ for all $i\neq i'$, and  
\begin{equation} 
\label{e:C(s) def}
C(s)=\bigcup_{i=1}^{m(s)} B(z_i(s), 2^{-k_n}).
\end{equation}  Let $c\in \{0,1\}$ and $t=s^\frown c$. For all $1\leq i\leq m(s)$ let $\ell_i(t)$ be the minimal integer such that $k_n\leq \ell_i(t) \leq c_{n+1} k_{n+1}$ and 
\begin{equation} \label{e:PL}
P_{\ell_i(t)}(B(z_i(s),2^{-k_n-1}))\geq 2^{\varphi(t)\ell_i(t)},
\end{equation} 
by \eqref{e:Pj} the numbers $\ell_i(t)$ are well-defined.
For every $1\leq i\leq m(s)$ let
\begin{equation}
\label{e:S_i(t) def}
\text{$S_i(t)$ be a $2^{-\ell_i(t)}$-packing of size $\lfloor 2^{\varphi(t)\ell_i(t)} \rfloor$ in $B(z_i(s),2^{-k_n-1})$}.
\end{equation}
Let $S(t)=\bigcup_{i=1}^{m(s)} S_i(t)$ and $m(t)=\#S(t)$, and let $z_j(t)\in K$ be the points for which $S(t)=\{z_j(t)\}_{1\leq j\leq m(t)}$. Define
\begin{equation} \label{e:ct} C(t)=\bigcup_{j=1}^{m(t)} B(z_j(t), 2^{-k_{n+1}}).
\end{equation}
Since $\varphi(t)\leq \alpha$ and $\ell_i(t)\leq c_{n+1} k_{n+1}$ for all $1\leq i\leq m(s)$, we have
\begin{equation} \label{e:mdef} 
m(t)\leq 2^{\alpha c_{n+1}k_{n+1}}m(s).
\end{equation} 
As $k_n\leq k_{n+1}-1$, we have
\begin{equation}
\label{e:ball is subset of bigger ball}
B(z,2^{-k_{n+1}})\subset B(z_i(s),2^{-k_n}) \text{ for all } 1\leq i\leq m(s) \text{ and } z\in S_i(t),
\end{equation}
and we obtain that $C(t)\subset C(s)$. For all $i\in \{1,\dots,m(s)\}$ let 
\begin{equation*} d_i(t)=\min\left\{\dist\left(B(z, 2^{-k_{n+1}}), B(z', 2^{-k_{n+1}})\right) : z, z' \in S_i(t),\, z \neq z'\right\},
\end{equation*}
where $\dist$ denotes the distance. As $\ell_i(t)\leq k_{n+1}-2$, for all $1\leq i\leq m(s)$ we obtain
\begin{equation} \label{e:dit}
d_i(t)>2^{-\ell_i(t)}-2^{-k_{n+1}+1}\geq 2^{-\ell_i(t)-1}>2^{-k_{n+1}}.    
\end{equation} 
In particular, inequality \eqref{e:dit} and the inductive hypothesis imply that the distance of the balls $B(z_j(t),2^{-k_{n+1}})$ and $B(z_{j'}(t),2^{-k_{n+1}})$ is greater than $2^{-k_{n+1}}$ for every $1\leq j<j'\leq m(t)$. Thus we defined $C(s)$ for all $s\in 2^{<\omega}$. For $x\in 2^{\omega}$ let
\begin{equation}
\label{e:C(x) def}
C(x)=\bigcap_{n=1}^{\infty} C(x\restriction n).
\end{equation}
Clearly, $C(x)$ is closed, hence compact. If $x\restriction n=y\restriction n$, then $C(x)$ and $C(y)$ are covered by the same balls of radius $2^{-k_n}$ which they both intersect, so the map $x\mapsto C(x)$ is continuous. Define 
\begin{equation*} 
\iC=\{C(x): x\in 2^{\omega}\}, 
\end{equation*}
then $\iC\subset \iK(K)$ is compact as a continuous image of the compact set $2^{\omega}$. 

In order to prove that $\{\dim_P C: C\in \iC\}=A$, let $x\in 2^{\omega}$ be arbitrarily fixed, it is enough to show that
\begin{equation} \label{e:PC}
\dim_P C(x)=\overline{\varphi}(x).
\end{equation} 
First we prove that $\dim_P C(x)\geq \overline{\varphi}(x)$.
Let $U\subset K$ be an arbitrary open set intersecting $C(x)$. Fix any large enough $n$ and $i\in \{1,\dots, m(x\restriction n)\}$ such that $B(z_i(x\restriction n),2^{-k_n})\subset U$. Let $s=x\restriction n$ and $t=x\restriction (n+1)$, then \eqref{e:ball is subset of bigger ball} and \eqref{e:S_i(t) def} imply that $B(z_i(s),2^{-k_n})$ contains $\lfloor 2^{\varphi(t)\ell_i(t)} \rfloor$ balls of the form $B(z,2^{-k_{n+1}})$, $z \in S_i(t)$. These balls are of pairwise distance greater than $2^{-\ell_i(t)-1}$ using \eqref{e:dit}, and all of them intersect $C(x)$ by the construction. Therefore
\begin{equation*} P_{\ell_i(t)+1}(C(x)\cap U)\geq \lfloor 2^{\varphi(t)\ell_i(t)}\rfloor.
\end{equation*} 
Hence $\overline{\dim}_B \, (C(x)\cap U)\geq \overline{\varphi}(x)$ by Fact~\ref{f:equiv}. As $U\subset K$ was an arbitrary open set,  Lemma~\ref{l:packing}~\eqref{i:i} implies $\dim_P C(x)\geq \overline{\varphi}(x)$. Therefore, in order to show \eqref{e:PC} it is enough to prove that
\begin{equation} \label{e:PP}
\overline{\dim}_B \, C(x)\leq \overline{\varphi}(x).
\end{equation}
Inequality \eqref{e:mdef} yields for all $n$ that 
\begin{equation} \label{e:mprod} 
m(x\restriction n)\leq \prod_{i=1}^n 2^{\alpha c_i k_i}=2^{o_n(1)k_n}, 
\end{equation}
where $o_n(1)\to 0$ easily follows from $\sum_{i=1}^{\infty} c_i<\infty$ and $k_n \uparrow \infty$.

We claim that for all $n\geq 0$, $i\in \{1,\dots,m(x\restriction n)\}$ and $\ell\in \{k_n,\dots,k_{n+1}\}$ for the ball $ B\defeq B(z_i(x\restriction n),2^{-k_n})$ we have
\begin{equation} \label{e:ncx}
N_{\ell}(C(x)\cap B)\leq 2^{\varphi(x\restriction (n+1))(\ell+1)}.
\end{equation}
Fix $n,i$ and let $s = x \restriction n $ and $t = x \restriction (n+1)$. Then \eqref{e:ct} yields
\begin{equation} \label{e:cxb} C(x)\cap  B=C(x)\cap \left(\bigcup_{z\in S_i(t)} B(z,2^{-k_{n+1}})\right).
\end{equation}
First assume that $\ell_i(t)-1\leq \ell\leq k_{n+1}$. Then \eqref{e:cxb} and \eqref{e:S_i(t) def} imply that 
\begin{equation*} N_{\ell}(C(x)\cap B)\leq \#S_i(t)\leq 2^{\varphi(t)(\ell+1)},
\end{equation*} 
so \eqref{e:ncx} holds. Now suppose $k_n\leq \ell<\ell_i(t)-1$ and define $B'=B(z_i(s),2^{-k_n-1})$. We claim that $P_{\ell}(C(x)\cap B)\leq P_{\ell+1}(B')$.
 Let $Z$ be any $2^{-\ell}$-packing in $C(x)\cap B$. Define $Z'\subset B'$ by simply keeping $Z\cap B'$ and replacing each $z\in Z\setminus B'$ with the unique $z'\in S_i(t)$ for which $z\in B(z',2^{-k_{n+1}})$. As $B(z,2^{-k_{n+1}})\cap B(z',2^{-k_{n+1}})=\emptyset$ for distinct points $z,z'\in S_i(t)$ and \eqref{e:cxb} holds, we obtain that $Z'$ is well-defined with $\#Z'=\#Z$, and $Z'\subset B'$ follows from \eqref{e:S_i(t) def}. The definition of $Z'$ and $\ell\leq k_{n+1}-2$ yield that any distance in $Z'$ is greater than $2^{-\ell}-2\cdot 2^{-k_{n+1}}\geq 2^{-(\ell+1)}$, so $Z'$ is a $2^{-(\ell+1)}$-packing in $B'$. Hence $P_{\ell}(C(x)\cap B)\leq P_{\ell+1}(B')$. Then Fact~\ref{f:equiv} and the minimality of $\ell_i(t)$ in \eqref{e:PL} with $\ell+1\leq \ell_i(t)$ imply that
\begin{equation*} 
N_{\ell}(C(x)\cap B)\leq P_{\ell}(C(x)\cap B)\leq P_{\ell+1}(B')\leq 2^{\varphi(t)(\ell+1)}, 
\end{equation*}
which proves \eqref{e:ncx}. 

We obtain $C(x)=\bigcup_{i=1}^{m(x\restriction n)} \left(C(x)\cap B\left(z_i(x \restriction n), 2^{-k_n}\right)\right)$ by \eqref{e:C(s) def} and \eqref{e:C(x) def}.
Thus the subadditivity of $N_{\ell}(\cdot)$, \eqref{e:ncx}, and \eqref{e:mprod} yield that for all $n\geq 0$ and $\ell\in \{k_n,\dots,k_{n+1}\}$ we have
\begin{align*} 
N_{\ell}(C(x))&\leq \sum_{i=1}^{m(x\restriction n)} N_{\ell}\left(C(x)\cap B\left(z_i(x \restriction n), 2^{-k_n}\right)\right) \\
&\leq m(x\restriction n) 2^{\varphi(x\restriction (n+1))(\ell+1)} \leq 2^{(\overline{\varphi}(x)+o_n(1))\ell}
\end{align*} 
with some $o_n(1)\to 0$, so \eqref{e:PP} holds. Therefore, the proof is complete if $A\subset [0,\alpha]$ for some $\alpha<\dim_P K$. 

Finally, we prove the general case. We may assume that $A\subset [0,\dim_P K)$ is non-empty, since if $A=\emptyset$ then $\iC=\emptyset$ works, and if $\dim_P K\in A$ and an appropriate family $\mathcal{C}$ for $A \setminus \{\dim_P K\}$ is constructed, then $\mathcal{C}\cup\{K\}$ works for $A$. Take a sequence $\alpha_n\uparrow \dim_P K$ with $\alpha_0\in A$, and define $A_n=A\cap [0,\alpha_n]$ for all $n\geq 0$. By Fact~\ref{f:dimK} we can choose $z_0\in K$ such that $\dim_P B(z_0,r)=\dim_P K$ for all $r>0$. By the first part of the proof we can choose a compact set $K_0\subset K$ with $\dim_P K_0=\alpha_0$, and we can pick compact families $\iC_n\subset \iK(B(z_0,2^{-n}))$ such that $\{\dim_P  C: C\in \iC_n\}=A_n$ for all $n\geq 0$. Clearly, we may assume that $z_0\in K_0$. Define
\begin{equation*}
\iC=\iC_0 \cup \{K_0\cup C: C\in \iC_n,~n\geq 1\}.
\end{equation*}	
It is straightforward to check that $\iC\subset \iK(K)$ is compact and $\{ \dim_P  C: C\in \iC\}=A$. The proof of the theorem is complete.	
\end{proof}

\section{Open problems} \label{s:open}

Our first problem is the most ambitious one. For the sake of simplicity we only formalize it in case of the Hausdorff dimension.

\begin{problem} \label{p:K} 
Let $K\subset \R^d$ be compact. Characterize the sets $A\subset [0,d]$ for which there exists a compact set $C\subset K$ such that $\{\dim_H E: E\in \iM_C\}=A$.
\end{problem}

J\"arvenp\"a\"a, J\"arvenp\"a\"a, Koivusalo, Li, Suomala, and Xiao \cite[Lemma~2.3]{JJKLSX} proved an analogue of Hawkes' theorem in complete metric spaces satisfying a mild doubling condition. Thus the proof of Theorem~\ref{t:compact family} probably works in compact metric spaces with a suitable doubling condition as well. As the case of arbitrary compact metric spaces seems to be out of reach with this method, we ask the following.

\begin{problem}
Let $K$ be a non-empty compact metric space and let $A \subset [0, \dim_H K]$ be an analytic set. Is there a compact set $\mathcal{C} \subset \mathcal{K}(K)$ with $\{\dim_H C : C \in \mathcal{C}\} = A$?
\end{problem}

As Mattila and Mauldin \cite[Theorem~7.5]{MM} proved that the packing dimension $\dim_P \colon \iK(\R^d)\to [0,d]$ is not Borel measurable, we do not know whether analogue versions of Theorems~\ref{t:characterization} and~\ref{t:compact family} hold for the packing dimension. We state the measurability problems as follows. 

\begin{problem} \label{p:meas} If $K\subset \R^d$ is a compact set, is the set $\{\dim_P E: E\in \iM_K\}$ analytic? If $\iC \subset \iK(\R^d)$ is compact, is the set $\{\dim_P C: C\in \iC\}$ analytic?
\end{problem}

If the answer to the above problem is negative, we can ask for a characterization.  
\begin{problem} \label{p:1} Let $d\geq 1$. Characterize the sets $A \subset [0, d]$ for which there exists a compact set $K \subset \R^d$ with $\{\dim_P E : E \in \iM_K\} = A$.
\end{problem}

\begin{problem} \label{p:2} Let $K\subset \R^d$ be compact. Characterize the sets $A \subset [0, \dim_P K]$ for which there exists a compact set $\iC\subset \iK(K)$ with $\{\dim_P C: C\in \iC\}=A$.
\end{problem}

\subsection*{Acknowledgments}
We are indebted to Jonathan M.~Fraser for some illuminating conversations and for providing us with the reference \cite{FWW}. We thank Ville Suomala for pointing our attention to the paper \cite{JJKLSX}. We also thank Ignacio Garc\'ia for pointing out that a result stated in an earlier version of the manuscript had already been known.

\end{document}